\RequirePackage{ifthen}
\newboolean{MPA}
\setboolean{MPA}{false}

\ifthenelse {\boolean{MPA}}
{
\documentclass[numbook]{svjour3} 
\smartqed
\usepackage[margin=1.3in]{geometry}
} {
\documentclass[11pt,letterpaper]{article}
\usepackage[margin=1in]{geometry}
}

\usepackage{graphicx} 

\usepackage{amsmath}
\usepackage{amssymb}
\usepackage{hyperref}
\usepackage{color}

\usepackage{soul}

\usepackage{enumitem}

\usepackage{forest}

\usepackage{tikz}
\usetikzlibrary{arrows, positioning, graphs}

\usepackage{enumitem}
\setlist{leftmargin=6mm,nolistsep,noitemsep}

\usepackage[capitalize,noabbrev]{cleveref}
\crefname{question}{Question}{Questions}
\crefname{step}{Step}{Steps}
\crefname{claim}{Claim}{Claims}
\crefname{problem}{Problem}{Problems}
\crefname{observation}{Observation}{Observations}
\crefname{remark}{Remark}{Remarks}

\DeclareMathOperator{\proj}{proj}

\DeclareMathOperator{\NA}{NA}

\DeclareMathOperator{\BPO}{BPO}

\DeclareMathOperator{\inc}{inc}
\DeclareMathOperator{\pri}{pri}

\DeclareMathOperator{\size}{size}

\newcommand{\tab}{\mathrm{Tab}}

\def\R{{\mathbb R}}
\def\Z{{\mathbb Z}}

\ifthenelse {\boolean{MPA}}
{

\newenvironment{prf}[1][]
{\begin{proof}}
{\qed \end{proof}}

\newenvironment{prfc}[1][]
{\begin{proof}{\emph{#1 }}}
{\qed \end{proof}}

\newcounter{claim} 
\renewenvironment{claim}[1][]
{\refstepcounter{claim} \begin{trivlist} \item[] {\bf Claim~\theclaim}\space#1 \itshape}
{\end{trivlist}}

\journalname{Mathematical Programming A}

} {

\usepackage{amsthm}
\newtheorem{theorem}{Theorem}
\newtheorem{proposition}{Proposition}
\newtheorem{lemma}{Lemma}

\newtheorem{corollary}{Corollary}
\newtheorem{remark}{Remark}

\newenvironment{prf}[1][]
{\begin{proof}}
{\end{proof}}

}

\newtheorem{observation}{Observation}

\allowdisplaybreaks

\definecolor{burntorange}{rgb}{0.8, 0.33, 0.0}
\definecolor{cadmiumgreen}{rgb}{0.0, 0.42, 0.24}

\newcommand{\var}{{\textcolor{blue}{x}}}
\newcommand{\vars}{{\textcolor{blue}{X}}}
\newcommand{\yvar}{{\textcolor{blue}{y}}}
\newcommand{\yvars}{{\textcolor{blue}{Y}}}

\newcommand{\dom}{{\textcolor{cadmiumgreen}{D}}}
\newcommand{\domI}{{\textcolor{cadmiumgreen}{D_I}}}

\newcommand{\con}{{\textcolor{burntorange}{c}}}
\newcommand{\cons}{{\textcolor{burntorange}{C}}}
\newcommand{\consin}{{\textcolor{burntorange}{C^\in}}}
\newcommand{\consge}{{\textcolor{burntorange}{C^\ge}}}
\newcommand{\consBPO}{{\textcolor{burntorange}{C_{\BPO}}}}
\newcommand{\consF}{{\textcolor{burntorange}{C_{F}}}}
\newcommand{\consinc}{{\textcolor{burntorange}{C^\in_\con}}}

\newcommand{\consgec}{{\textcolor{burntorange}{C^\ge_\con}}}
\newcommand{\consL}{{\textcolor{burntorange}{C_L}}}
\newcommand{\consLge}{{\textcolor{burntorange}{C^\ge_L}}}
\newcommand{\consLle}{{\textcolor{burntorange}{C^\le_L}}}
\newcommand{\consS}{{\textcolor{burntorange}{C_S}}}
\newcommand{\consSge}{{\textcolor{burntorange}{C^\ge_S}}}
\newcommand{\consSle}{{\textcolor{burntorange}{C^\le_S}}}

\newcommand{\val}{\nu}
\newcommand{\wei}{\omega}
\newcommand{\pwidth}{{w_{\proj}}}
\newcommand{\pswidth}{w_{\text{ps}}}
\newcommand{\itwidth}{{w_{\inc}}}
\newcommand{\ptwidth}{{w_{\pri}}}

\newcommand{\ass}{\tau}
\newcommand{\assb}{\tau'}
\newcommand{\assc}{\tau''}
\newcommand{\Ass}{\hat X}

\newcommand{\pairs}{\mathcal P}

\DeclareMathOperator{\cb}{cb}

\newcommand{\pare}[1]{\left(#1\right)}
\newcommand{\bra}[1]{\left\{#1\right\}}
\newcommand{\card}[1]{\left|#1\right|}

\usepackage[title]{appendix}

\begin{document}

\title{Projection-width as a structural parameter for discrete separable optimization}

\ifthenelse {\boolean{MPA}}
{

\author{Alberto Del Pia}
\institute{Alberto~Del~Pia \at
              Department of Industrial and Systems Engineering 
              \& Wisconsin Institute for Discovery \\
              University of Wisconsin-Madison, Madison, WI, USA \\
              \email{delpia@wisc.edu}}
}
{
\author{Alberto Del Pia
\thanks{Department of Industrial and Systems Engineering \& Wisconsin Institute for Discovery,
             University of Wisconsin-Madison, Madison, WI, USA.
             E-mail: {\tt delpia@wisc.edu}.}}
}

\date{February 5, 2026}


\maketitle

\begin{abstract}
While several classes of integer linear optimization problems are known to be solvable in polynomial time, far fewer tractability results exist for integer nonlinear optimization.
In this work, we narrow this gap by identifying a broad class of discrete nonlinear optimization problems that admit polynomial-time algorithms.
Central to our approach is the notion of \emph{projection-width,} a structural parameter for systems of separable constraints, defined via branch decompositions of variables and constraints. 
We show that several fundamental discrete optimization and counting problems can be solved in polynomial time when the projection-width is polynomially bounded, including optimization, counting, top-$k$, and weighted constraint violation problems.
Our results subsume and generalize some of the strongest known tractability results across multiple research areas: integer linear optimization, binary polynomial optimization, and Boolean satisfiability. 
Although these results originated independently within different communities and for seemingly distinct problem classes, our framework unifies and significantly generalizes them under a single structural perspective.

\ifthenelse {\boolean{MPA}}
{
} {}
\end{abstract}

\ifthenelse {\boolean{MPA}}
{}{
}

\section{Introduction}
\label{sec intro}

The field of integer linear optimization has reached a high level of maturity, with a rich theoretical framework, efficient algorithms, and numerous applications.
A central outcome of this development is the identification of broad and structurally rich classes of integer linear optimization problems that admit polynomial-time algorithms (see \cite{SchBookIP,ConCorZamBook} and references therein).
Many optimization problems arising in applications, however, are inherently nonlinear. 
Despite their importance, integer nonlinear optimization problems remain far less understood from a complexity-theoretic perspective, with polynomial-time solvability known only for a few highly restrictive classes.
The main goal of this paper is to narrow this gap by identifying a broad class of integer nonlinear optimization problems that can be solved in polynomial time.

\paragraph{Separable systems.}

The framework that we introduce in this paper operates over highly general systems of separable constraints.
A \emph{separable system} is a quadruple $\pare{\vars,\dom,\consin,\consge}$, where $\vars$ is a finite set of \emph{variables}, $\dom$ is a finite \emph{domain} set, and where $\consin$, $\consge$ are sets of \emph{separable constraints} of the form
\begin{align*}
& \sum_{\var \in \vars} f^\con_\var(\var) \in \Delta^\con && \con \in \consin, \\
& \sum_{\var \in \vars} f^\con_\var(\var) \ge \delta^\con && \con \in \consge,
\end{align*}
where $f^\con_\var : \dom \to \Z$ for every $\var \in \vars$ and $\con \in \consin \cup \consge$, 
where $\delta^\con \in \Z$ for every $\con \in \consge$,
and where $\Delta^\con$ is a finite subset of $\Z$ for every $\con \in \consin$.
Throughout the paper, we assume that a separable system is given by explicitly providing $f^\con_\var$ for every $\var \in \vars$ and $\con \in \cons$, $\Delta^\con$ for every $\con \in \consin$, and $\delta^\con$ for every $\con \in \consge$.
Clearly, a constraint in $\consge$ can also be expressed as a constraint in $\consin$.
However, we choose to handle constraints in $\consge$ separately, as doing so allows us to efficiently solve broader classes of problems.
Separable constraints constitute a broad and expressive class of constraints.
Notable special cases include linear and pseudo-Boolean inequalities and set constraints, as well as parity or, more generally, modulo constraints.

\paragraph{The problems.}

In this paper, we consider several fundamental discrete optimization and counting problems defined over a separable system $\pare{\vars,\dom,\consin,\consge}$, which we informally introduce below:
\begin{enumerate}
\item
\emph{Optimization:} Given a separable value function for each variable, find a highest-value assignment from $\vars$ to $\dom$ satisfying all the constraints. 
(See \cref{sec opt} and \cref{th optimization}.)
\item
\emph{Counting:} Count the number of assignments from $\vars$ to $\dom$ satisfying all the constraints.
(See \cref{sec counting} and \cref{th counting}.)
\item
\emph{Top-$k$:} Given a separable value function for each variable, find $k$ highest-value assignments from $\vars$ to $\dom$ satisfying all the constraints. 
(See \cref{sec top-k} and \cref{th top-k}.)
\item
\emph{Weighted constraint violation:} Given a weight for each constraint, find an assignment from $\vars$ to $\dom$ that minimizes the weighted violation of the constraints.
(See \cref{sec weighted constraint violation} and \cref{th weighted constraint violation}.)
\end{enumerate}

Although the four problems above already form a substantial class of fundamental discrete nonlinear optimization and counting problems, our framework can be applied more broadly to other problems defined over separable systems.
For instance, it can be leveraged in polyhedral theory to derive compact extended formulations using the approach of \cite{MarRarCam90}, and in knowledge compilation to translate constraints into succinct deterministic DNNF representations, as in~\cite{Cap16PhD}. 

\paragraph{Our results.}

We present exact algorithms for the four problems mentioned above.
The running times of our algorithms are low-degree polynomials in $\card{\vars}$, $\card{\dom}$, $\card{\cons}$, where $\cons := \consin \cup \consge$, and in two more parameters: $\pwidth$ and $\Lambda$.
The first one, $\pwidth$, is the \emph{projection-width} of the separable system, which is the central concept introduced in this paper and is formally defined in \cref{sec def projection-width}.
This concept is inspired by the \emph{PS-width} of CNF formulas, originally developed for satisfiability problems in~\cite{SaeTelVat14}, and it generalizes the idea in a substantially different context and captures structural properties specific to separable systems over finite domains.
The second parameter is $\Lambda$, which denotes the maximum time required to check whether an integer belongs to a set $\Delta^\con$, for some $\con \in \consin$.
This parameter is introduced for convenience and generality, since $\Lambda$ depends on how $\Delta^c$ is given.
The running time that we obtain for the optimization and the counting problems is
\begin{equation*}
O\pare{
\pwidth^3 \pare{\card{\vars}+\card{\cons}} \card{\cons}
+
\pwidth \card{\consin} \Lambda
+
\card{\vars} \card{\cons} \card{\dom}\log\pare{\card{\dom}}}.
\end{equation*}
For the top-$k$ problem, we get 
\begin{equation*}
O\pare{
\pwidth^3 \pare{\card{\vars}+\card{\cons}} \pare{\card{\cons} + k \log(k)}
+
\pwidth \card{\consin} \Lambda
+
\card{\vars} \card{\cons} \card{\dom}\log\pare{\card{\dom}}},
\end{equation*}
and for the weighted constraint violation we get 
\begin{equation*}
O\pare{
\pwidth^3 \pare{\card{\vars}+\card{\cons}} \card{\cons}
+
\card{\vars} \card{\cons} \card{\dom}\log\pare{\card{\dom}}}.
\end{equation*}
We refer the reader to \cref{th optimization,th counting,th top-k,th weighted constraint violation} for the precise statements of our results.
All running times are stated under the \emph{arithmetic model} of computation, in which each basic arithmetic operation (addition, subtraction, multiplication, division, and comparison) is assumed to take constant time, independent of operand sizes.
Moreover, in our algorithms, the sizes of all intermediate and output values are bounded by a polynomial in the input size.
We note that sharper running-time bounds may be attainable through a more refined analysis, especially in specific cases.
For instance, when the involved functions are linear rather than separable, when $\card{\dom} = 2$, or when branch decompositions (introduced later in this section) are linear~\cite{SaeTelVat14}.

\medskip

Our results establish a broad class of tractable discrete nonlinear optimization and counting problems.
To illustrate their generality, we show that they subsume some of the strongest known tractability results in integer linear optimization, binary polynomial optimization, and Boolean satisfiability.
In integer linear optimization, it subsumes the polynomial-time solvability of instances with bounded incidence treewidth established in~\cite{GanOrdRam17}.
In binary polynomial optimization, it generalizes the tractability of instances defined on hypergraphs with bounded incidence treewidth shown in~\cite{CapdPDiG24xx}, while additionally allowing the presence of further constraints.
In Boolean satisfiability, our framework encompasses the tractability of weighted MaxSAT and \#SAT on CNF formulas with bounded PS-width~\cite{SaeTelVat14}, which captures nearly all known tractable cases of these problems~\cite{Cap16PhD}.

All these three results follow as direct corollaries of our framework under highly specialized and restrictive assumptions.
Compared to our general setting, each of them is considerably narrower.
In particular, all functions involved are linear,
only inequality constraints are considered,
and the variable domains are restricted to 
integer intervals in integer linear optimization, and to $\bra{0,1}$ in binary polynomial optimization and Boolean satisfiability.
While the primary contribution of this paper lies in the nonlinear setting, the fact that these disparate results emerge as special cases highlights the broader significance of projection-width, which provides a common structural explanation for tractability across seemingly different problem classes and research communities.

In the remainder of this section, we present the definition of projection-width.

\subsection{Projection-width}
\label{sec def projection-width}

\paragraph{Translated separable system.}

Consider a separable system $\pare{\vars,\dom,\consin,\consge}$, and let $\cons := \consin \cup \consge$.
Although not strictly necessary, we begin by rewriting the constraints in~$\cons$.
This reformulation clarifies the intuition behind projection-width and simplifies the notation used throughout the paper.

For every $\con \in \cons$ we define 
\begin{align*}
g^\con_\var & := 
f^\con_\var - \min\bra{f^\con_\var(d) \mid d \in \dom} 
& \forall \var \in \vars.
\end{align*}
For every $\con \in \consin$, we also define
\begin{align*}
\Gamma^\con & := \Delta^\con - \sum_{\var \in \vars} \min\bra{f^\con_\var(d) \mid d \in \dom},
\end{align*}
and for every $\con \in \consge$, we set
\begin{align*}
\gamma^\con & := \delta^\con - \sum_{\var \in \vars} \min\bra{f^\con_\var(d) \mid d \in \dom}.
\end{align*}
We then obtain the equivalent \emph{translated separable system}
\begin{align*}
& \sum_{\var \in \vars} g^\con_\var(\var) \in \Gamma^\con && \con \in \consin, \\
& \sum_{\var \in \vars} g^\con_\var(\var) \ge \gamma^\con && \con \in \consge.
\end{align*}
We now have $g^\con_\var : \dom \to \Z_{\ge 0}$ and $\min\bra{g^\con_\var(d) \mid d \in \dom} = 0$ for every $\var \in \vars$ and $\con \in \cons$.
As a result, we can assume $\Gamma^\con \subseteq \Z_{\ge 0}$, for every $\con \in \consin$.
Similarly, we can assume $\gamma^\con \ge 0$ for every $\con \in \consge$, since otherwise every assignment satisfies the constraint, and so the constraint can be discarded.
For every $\con \in \consin$, we define $\gamma^\con$ as the largest element in $\Gamma^\con$.
From now on, we will mostly consider the translated separable system instead of the original separable system.

\paragraph{Projections.}

For constraint $\con \in \cons$, $\vars' \subseteq \vars$, and $\ass : \vars' \to \dom$, we define the \emph{constraint bound} of $\con$ by $\ass$ as
\begin{align*}
\cb^\con\pare{\ass} & := \sum_{\var \in \vars'} g^\con_\var(\ass(\var)).
\end{align*}
Note that $\cb^\con\pare{\ass} \in \Z_{\ge 0}$.
To build intuition, we observe that the constraint bound 
provides the minimum possible left hand side of constraint $\con$ over all assignments from $\vars$ to $\dom$ whose restriction to $\vars'$ is $\ass$.

Given $\vars' \subseteq \vars$, $\cons' \subseteq \cons$, and $\ass : \vars' \to \dom$, we denote by $\cons' / \ass$ the map from $\cons'$ to $\Z_{\ge 0}$ defined by
\begin{align*}
& \pare{\cons' / \ass}^\con := \min\bra{\cb^\con\pare{\ass},\gamma^\con} && \forall \con \in \cons'.
\end{align*}
Note that in this context, we use the word ``map'' to avoid confusion with our ``assignments.''
Since $\gamma^\con \in \Z_{\ge 0}$ and $\cb^\con\pare{\ass} \in \Z_{\ge 0}$, we have $\pare{\cons' / \ass}^\con \in \bra{0,1,\dots,\gamma^\con}$, for every $\con \in \cons'$.
Note that, if $\cons' = \emptyset$, then $\cons' / \ass$ is the empty map $\epsilon : \emptyset \to \Z_{\ge 0}$.
On the other hand, if $\vars' = \emptyset$, then $\ass$ is the empty assignment $\epsilon : \emptyset \to \dom$, and $\cons' / \epsilon$ is given explicitly by
\begin{align*}
&\pare{\cons' / \epsilon}^\con = 0 && \forall \con \in \cons'.
\end{align*}
We observe that $\cons' / \ass$ captures the contribution of the assignment $\ass$ to the left hand sides of the constraints in $\cons'$, up to the threshold $\gamma^\con$.

Given $\vars' \subseteq \vars$ and $\cons' \subseteq \cons$, we define
\begin{align*}
\proj(\cons', \vars'): = \bra{ \cons' / \ass \mid \ass : \vars' \to \dom}.
\end{align*}
As a result, $\proj(\cons', \vars')$ encodes how all assignments from $\vars'$ to $\dom$ can contribute to the left hand sides of the constraints in $\cons'$, before the threshold $\gamma^\con$ is reached.


\paragraph{Projection-width.}

Consider now a branch decomposition $T$ of $\vars \cup \cons$, which is a rooted binary tree with a one-to-one correspondence between the leaves of $T$ and the set $\vars \cup \cons$.
We say that a vertex of $T$ is an \emph{inner vertex} if it has children, and is a \emph{leaf} if it has no children.
Given a vertex $v$ of $T$, we denote by $T_v$ the subtree of $T$ rooted in $v$, by $\cons_v$ the set of constraints in $\cons$ such that the corresponding vertex appears in the leaves of $T_v$, and by $\vars_v$ the set of variables in $\vars$ that similarly appear in the leaves of $T_v$.
We denote by $\overline{\cons_v}:=\cons \setminus \cons_v$ and by $\overline{\vars_v}:=\vars \setminus \vars_v$. 
A key role is played by the two projections
\begin{align*}
\proj(v) & := \proj(\overline{\cons_v}, \vars_v) 
= \bra{ \overline{\cons_v} / \ass \mid \ass : \vars_v \to \dom}, \\
\proj(\overline v) & := \proj(\cons_v, \overline{\vars_v})
= \bra{ \cons_v / \ass \mid \ass : \overline{\vars_v} \to \dom}.
\end{align*}
If we denote by $V(T)$ the set of vertices of $T$, the \emph{projection-width} of the separable system $\pare{\vars,\dom,\consin,\consge}$ and $T$ is defined by
\begin{align*}
\max _{v \in V(T)} 
\max \pare{\card{\proj(v)}, \card{\proj(\overline v)}}.
\end{align*}
The \emph{projection-width} of the separable system $\pare{\vars,\dom,\consin,\consge}$ is then defined as the minimum among the projection-widths of the separable system $\pare{\vars,\dom,\consin,\consge}$ and $T$, over all branch decompositions $T$ of $\vars \cup \cons$.

\paragraph{Organization of the paper.}
The rest of the paper is organized as follows.
In~\cref{sec projections}, we study the relationship between the projections $\proj(v)$ and $\proj(\overline v)$ in a branch decomposition, and use these insights to efficiently construct all such sets throughout the decomposition.
In~\cref{sec shapes}, we introduce the notion of \emph{shapes}, which enables us to retain in the branch decomposition only the assignments that satisfy all the constraints, effectively filtering out the others.
In~\cref{sec algorithms}, we present and analyze our algorithms for optimization, top-$k$, counting, and weighted constraint violation.
Finally, in~\cref{sec consequences}, we obtain some corollaries of our main theorems, and we discuss how our results subsume previously known tractability results in integer linear optimization, binary polynomial optimization, and Boolean satisfiability.

\section{Projections}
\label{sec projections}

In the following, given $\vars_1,\vars_2$ disjoint subsets of $\vars$, $\ass_1 : \vars_1 \to \dom$, $\ass_2 : \vars_2 \to \dom$, we denote by $\ass_1 \cup \ass_2$ the assignment from $\vars_1 \cup \vars_2$ to $\dom$ defined by
\begin{align*}
\pare{\ass_1 \cup \ass_2}(\var) := 
\begin{cases}
\ass_1(\var) & \text{if $\var \in \vars_1$}, \\
\ass_2(\var) & \text{if $\var \in \vars_2$}.
\end{cases}
\end{align*}

\begin{observation}
\label{obs vb}
Consider a separable system $\pare{\vars,\dom,\consin,\consge}$, and let $\con \in \consin \cup \consge$.
The following properties hold:
\begin{enumerate}[label=(\roman*)]
\item
\label{obs vb 1}
Given $\vars' \subseteq \vars$ and $\ass : \vars' \to \dom$,
$\cb^\con\pare{\ass} \in \Z_{\ge 0}$.
\item
\label{obs vb 3}
Given $\vars_1,\vars_2$ disjoint subsets of $\vars$, $\ass_1 : \vars_1 \to \dom$, $\ass_2 : \vars_2 \to \dom$, we have $\cb^\con\pare{\ass_1 \cup \ass_2} = \cb^\con\pare{\ass_1} + \cb^\con\pare{\ass_2}$.
\end{enumerate}
\end{observation}

\begin{prf}
\ref{obs vb 1}.
Follows directly from $g^\con_\var(\ass(\var)) \in \Z_{\ge 0}$ for every $\var \in \vars$.

\ref{obs vb 3}.
Follows directly from the definition of constraint bound.
\end{prf}

We will often use the next simple result.
The first few times we will employ it, we will use it with $c=0$.
However, later on we will need the more general version presented below.
\begin{lemma}
\label{lem 3min}
For all nonnegative $a,b,c,\gamma$, the following identity holds:
\begin{align*}
\min\bra{a+b,\gamma - c} = \min\bra{a+\min\bra{b,\gamma},\gamma - c}.
\end{align*}
\end{lemma}

\begin{prf}
\begin{align*}
\min\bra{a+\min\bra{b,\gamma},\gamma - c}
& = 
\min\bra{a+b,a+\gamma,\gamma - c} \\
& = \min\bra{a+b,\gamma - c}.
\end{align*}
\end{prf}

\subsection{Structure of $\proj(v)$}

The next result is at the heart of the relationship between sets $\proj(v)$ in a branch decomposition.
\begin{lemma}
\label{lem phi ass}
Consider a separable system $\pare{\vars,\dom,\consin,\consge}$, let $\cons := \consin \cup \consge$, let $T$ be a branch decomposition of $\vars \cup \cons$, and let $v$ be an inner vertex of $T$ with children $v_1$ and $v_2$. 
Let $\ass : \vars_v \to \dom$, and let $\ass_1$ and $\ass_2$ denote the restrictions of $\ass$ to $\vars_{v_1}$ and $\vars_{v_2}$, respectively.
Then,
\begin{align*}
\pare{\overline{\cons_v} / \ass}^\con
=
\min\bra{\pare{\overline{\cons_{v_1}} / \ass_1}^\con + \pare{\overline{\cons_{v_2}} / \ass_2}^\con, \gamma^\con}
\qquad \forall \con \in \overline{\cons_v}.
\end{align*}
\end{lemma}

\begin{prf}
Let $\con \in \overline{\cons_v}$ and observe that $\overline{\cons_v} = \overline{\cons_{v_1}} \cap \overline{\cons_{v_2}}$. 
We have 
\begin{align*}
\pare{\overline{\cons_v} / \ass}^\con
& =
\min\bra{\cb^\con\pare{\ass},\gamma^\con} \\
& =
\min\bra{\cb^\con\pare{\ass_1}+\cb^\con\pare{\ass_2},\gamma^\con} 
&& \text{(from \cref{obs vb} \ref{obs vb 3})} \\
& =
\min\bra{\min\bra{\cb^\con\pare{\ass_1},\gamma^\con} + \min\bra{\cb^\con\pare{\ass_2},\gamma^\con}, \gamma^\con} 
&& \text{(from \cref{lem 3min})} \\
& =
\min\bra{\pare{\overline{\cons_{v_1}} / \ass_1}^\con + \pare{\overline{\cons_{v_2}} / \ass_2}^\con, \gamma^\con}.
\end{align*}
\end{prf}

\cref{lem phi ass} suggests the following relationship between $\Phi \in \proj(v)$, $\Phi_1 \in \proj(v_1)$, and $\Phi_2 \in \proj(v_2)$:
\begin{align}
\label{cond l1}
\tag{L1}
\Phi^\con & = \min\bra{\Phi_1^\con + \Phi_2^\con, \gamma^\con}
&& \forall \con \in \overline{\cons_v}.
\end{align}

\begin{lemma}
\label{lem phi proj}
Consider a separable system $\pare{\vars,\dom,\consin,\consge}$, let $\cons := \consin \cup \consge$, let $T$ be a branch decomposition of $\vars \cup \cons$, and let $v$ be an inner vertex of $T$ with children $v_1$ and $v_2$. 
If $\Phi \in \proj(v)$, then there exist $\Phi_1 \in \proj(v_1)$ and $\Phi_2 \in \proj(v_2)$ such that \eqref{cond l1} holds. 
Vice versa, if $\Phi_1 \in \proj(v_1)$ and $\Phi_2 \in \proj(v_2)$, then there is a unique $\Phi \in \proj(v)$ such that \eqref{cond l1} holds.
\end{lemma}

\begin{prf}
Let $\Phi \in \proj(v)$.
Then, there exists $\ass : \vars_v \to \dom$ such that $\Phi = \overline{\cons_v} / \ass$.
Observe that $\vars_v$ is the disjoint union of sets $\vars_{v_1}$ and $\vars_{v_2}$.
Let $\ass_1$ and $\ass_2$ denote the restrictions of $\ass$ to $\vars_{v_1}$ and $\vars_{v_2}$, respectively.
Let $\Phi_1 := \overline{\cons_{v_1}} / \ass_1 \in \proj(v_1)$ and $\Phi_2 := \overline{\cons_{v_2}} / \ass_2 \in \proj(v_2)$.
Then, \cref{lem phi ass} implies that \eqref{cond l1} holds.

Let $\Phi_1 \in \proj(v_1)$ and $\Phi_2 \in \proj(v_2)$.
Then, there exist $\ass_1 : \vars_{v_1} \to \dom$ and $\ass_2 : \vars_{v_2} \to \dom$ such that $\Phi_1 = \overline{\cons_{v_1}} / \ass_1$ and $\Phi_2 = \overline{\cons_{v_2}} / \ass_2$.
Let $\ass := \ass_1 \cup \ass_2 : \vars_v \to \dom$, and define $\Phi := \overline{\cons_v} / \ass \in \proj(v)$.
Then, \cref{lem phi ass} implies that \eqref{cond l1} holds.
From \eqref{cond l1}, $\Phi$ is clearly unique.
\end{prf}

\subsection{Structure of $\proj(\overline{v})$}

The relationship between sets $\proj(\overline{v})$ in a branch decomposition is slightly more complex, since it also involves maps from $\proj(v)$.
The next result is at the heart of this relationship.

\begin{lemma}
\label{lem psi ass}
Consider a separable system $\pare{\vars,\dom,\consin,\consge}$, let $\cons := \consin \cup \consge$, let $T$ be a branch decomposition of $\vars \cup \cons$, and let $v$ be an inner vertex of $T$ with children $v_1$ and $v_2$. 
Let $\ass_1 : \overline{\vars_{v_1}} \to \dom$, and let $\ass$ and $\ass_2$ denote the restrictions of $\ass_1$ to $\overline{\vars_v}$ and $\vars_{v_2}$, respectively.
Then,
\begin{align*}
\pare{\cons_{v_1} / \ass_1}^\con
=
\min\bra{\pare{\cons_{v} / \ass}^\con + \pare{\overline{\cons_{v_2}} / \ass_2}^\con, \gamma^\con} 
\qquad \forall \con \in \cons_{v_1}.
\end{align*}
Symmetrically, let $\ass_2 : \overline{\vars_{v_2}} \to \dom$, and let $\ass$ and $\ass_1$ denote the restrictions of $\ass_2$ to $\overline{\vars_v}$ and $\vars_{v_1}$, respectively.
Then, 
\begin{align*}
\pare{\cons_{v_2} / \ass_2}^\con
=
\min\bra{\pare{\cons_{v} / \ass}^\con + \pare{\overline{\cons_{v_1}} / \ass_1}^\con, \gamma^\con} 
\qquad \forall \con \in \cons_{v_2}.
\end{align*}
\end{lemma}

\begin{prf}
We only prove the first part of the statement, since the second part is symmetric.
Let $\con \in \cons_{v_1}$ and observe that $\cons_{v_1} = \cons_v \cap \overline{\cons_{v_2}}$.
We have
\begin{align*}
\pare{\cons_{v_1} / \ass_1}^\con
& =
\min\bra{\cb^\con\pare{\ass_1},\gamma^\con} \\
& =
\min\bra{\cb^\con\pare{\ass}+\cb^\con\pare{\ass_2},\gamma^\con} 
&& \text{(from \cref{obs vb} \ref{obs vb 3})} \\
& = 
\min\bra{\min\bra{\cb^\con\pare{\ass},\gamma^\con} + \min\bra{\cb^\con\pare{\ass_2},\gamma^\con}, \gamma^\con}
&& \text{(from \cref{lem 3min})} \\
& = 
\min\bra{\pare{\cons_{v} / \ass}^\con + \pare{\overline{\cons_{v_2}} / \ass_2}^\con, \gamma^\con}.
\end{align*}
\end{prf}

\cref{lem psi ass} suggests two more relationships.
The first one between $\Psi \in \proj(\overline{v})$, $\Psi_1 \in \proj(\overline{v_1})$, and $\Phi_2 \in \proj(v_2)$:
\begin{align}
\label{cond l2}
\tag{L2}
\Psi_1^\con & = \min\bra{\Psi^\con + \Phi_2^\con, \gamma^\con} && \forall \con \in \cons_{v_1}.
\end{align}
The second one between $\Psi \in \proj(\overline{v})$, $\Phi_1 \in \proj(v_1)$, and $\Psi_2 \in \proj(\overline{v_2})$:
\begin{align}
\label{cond l3}
\tag{L3}
\Psi_2^\con & = \min\bra{\Psi^\con + \Phi_1^\con, \gamma^\con} && \forall \con \in \cons_{v_2}.
\end{align}

\begin{lemma}
\label{lem psi proj}
Consider a separable system $\pare{\vars,\dom,\consin,\consge}$, let $\cons := \consin \cup \consge$, let $T$ be a branch decomposition of $\vars \cup \cons$, and let $v$ be an inner vertex of $T$ with children $v_1$ and $v_2$. 
If $\Psi_1 \in \proj(\overline{v_1})$, then there exist $\Psi \in \proj(\overline{v})$ and $\Phi_2 \in \proj(v_2)$ such that \eqref{cond l2} holds.
Vice versa, if $\Psi \in \proj(\overline{v})$ and $\Phi_2 \in \proj(v_2)$, then there is a unique $\Psi_1 \in \proj(\overline{v_1})$ such that \eqref{cond l2} holds.
Symmetrically, if $\Psi_2 \in \proj(\overline{v_2})$, then there exist $\Psi \in \proj(\overline{v})$ and $\Phi_1 \in \proj(v_1)$ such that \eqref{cond l3} holds.
Vice versa, if $\Psi \in \proj(\overline{v})$ and $\Phi_1 \in \proj(v_1)$, then there is a unique $\Psi_2 \in \proj(\overline{v_2})$ such that \eqref{cond l3} holds.
\end{lemma}

\begin{prf}
We only prove the first part of the statement, since the second part is symmetric.

Let $\Psi_1 \in \proj(\overline{v_1})$.
Then, there exists $\ass_1 : \overline{\vars_{v_1}} \to \dom$ such that $\Psi_1 = \cons_{v_1} / \ass_1$.
Observe that $\overline{\vars_{v_1}}$ is the union of disjoint sets $\overline{\vars_v}$ and $\vars_{v_2}$.
Let $\ass$ and $\ass_2$ denote the restrictions of $\ass_1$ to $\overline{\vars_v}$ and $\vars_{v_2}$, respectively.
Let $\Psi := \cons_{v} / \ass \in \proj(\overline{v})$ and $\Phi_2 := \overline{\cons_{v_2}} / \ass_2 \in \proj(v_2)$.
Then, \cref{lem psi ass} implies that \eqref{cond l2} holds.

Let $\Psi \in \proj(\overline{v})$ and $\Phi_2 \in \proj(v_2)$.
Then, there exist $\ass : \overline{\vars_v} \to \dom$ and $\ass_2 : \vars_{v_2} \to \dom$ such that $\Psi = \cons_v / \ass$ and $\Phi_2 = \overline{\cons_{v_2}} / \ass_2$.
Let $\ass_1 := \ass \cup \ass_2 : \overline{\vars_{v_1}} \to \dom$, and define $\Psi_1 := \cons_{v_1} / \ass_1 \in \proj(\overline{v_1})$.
Then, \cref{lem psi ass} implies that \eqref{cond l2} holds.
From \eqref{cond l2}, $\Psi_1$ is clearly unique.
\end{prf}

\subsection{Constructing all projections}

In this section, we use the structural results in \cref{lem phi proj,lem psi proj} to construct efficiently all sets $\proj(v)$ and $\proj(\overline v)$ in a branch decomposition.
We begin with some projections that have a very simple structure.

\begin{remark}
\label{rem proj}
Consider a separable system $\pare{\vars,\dom,\consin,\consge}$, let $\cons := \consin \cup \consge$, and let $T$ be a branch decomposition of $\vars \cup \cons$.
\begin{enumerate}[label=(\roman*)]
\item
\label{rem proj constraint-leaf}
\ul{$\proj(l)$, for a leaf $l$ of $T$ corresponding to a constraint $\con \in \cons$.}
We have $\vars_l = \emptyset$, thus there is only one assignment from $\vars_l$ to $\dom$, the empty assignment $\epsilon$.
Therefore, $\proj(l)$ contains only one map.
Because $\overline{\cons_l} = \cons \setminus \bra{\con}$, this map is given by $(\cons \setminus \bra{\con}) / \epsilon$.

\item
\label{rem proj variable-leaf}
\ul{$\proj(l)$, for a leaf $l$ of $T$ corresponding to a variable $\var \in \vars$.}
We have $\vars_l = \bra{\var}$, thus there are $\card{\dom}$ assignment from $\vars_l$ to $\dom$.
Therefore, $\proj(l)$ contains at most $\card{\dom}$ maps.
Because $\overline{\cons_l} = \cons$, they are of the form $\cons / \ass$, for every assignment $\ass : \bra{\var} \to \dom$.

\item
\label{rem proj root phi}
\ul{$\proj(r)$, for the root $r$ of $T$.}
Since $\overline{\cons_r} = \emptyset$, $\proj(r)$ contains only the empty map $\epsilon : \emptyset \to \Z_{\ge 0}$.

\item
\label{rem proj root psi}
\ul{$\proj(\overline{r})$, for the root $r$ of $T$.}
Since $\overline{\vars_r} = \emptyset$, $\proj(\overline{r})$ contains only one map.
Because $\cons_r = \cons$, this map is given by $\cons / \epsilon$.
\end{enumerate}
\end{remark}


\begin{proposition}
\label{prop all proj}
Consider a separable system $\pare{\vars,\dom,\consin,\consge}$, let $\cons := \consin \cup \consge$, and let $T$ be a branch decomposition of $\vars \cup \cons$ of projection-width $\pwidth$.
There is an algorithm that computes $\proj(v)$ and $\proj(\overline{v})$, for every vertex $v$ of $T$, in time 
\begin{align*}
O\pare{\pwidth^2 \log(\pwidth) \pare{\card{\vars}+\card{\cons}} \card{\cons} + \card{\vars}\card{\cons}\card{\dom}\log\pare{\card{\dom}}}.
\end{align*}
\end{proposition}

\begin{prf}
The construction of the sets $\proj(v)$, for every vertex $v$ of $T$, is performed in a bottom up manner.

\textbf{Leaves corresponding to constraints.}
For every leaf $l$ of $T$ corresponding to a constraint $\con \in \cons$, we know from \cref{rem proj} \ref{rem proj constraint-leaf} that $\proj(l)$ contains only one map, and it can be constructed in time $O\pare{\card{\cons}}$.
Since there are $\card{\cons}$ leaves corresponding to constraints, they require $O(\card{\cons}^2)$ time.

\textbf{Leaves corresponding to variables.}
For every leaf $l$ of $T$ corresponding to a variable $\var \in \vars$, we know from \cref{rem proj} \ref{rem proj variable-leaf} that $\proj(l)$ contains at most $\card{\dom}$ maps, and they can be constructed, possibly with duplicates, in time $O\pare{\card{\cons}\card{\dom}}$.
We sort the obtained $\card{\dom}$ maps lexicographically in time $O\pare{\card{\cons}\card{\dom}\log\pare{\card{\dom}}}$, and then we delete duplicates in time $O\pare{\card{\cons}\card{\dom}}$.
Therefore, for every leaf $l$ of $T$, the set $\proj(l)$ can be constructed in time $O\pare{\card{\cons}\card{\dom}\log\pare{\card{\dom}}}$.
Since there are $\card{\vars}$ leaves corresponding to variables, they require $O\pare{\card{\vars}\card{\cons}\card{\dom}\log\pare{\card{\dom}}}$ time.

\textbf{Inner vertices.}
Consider an inner vertex $v$ of $T$ with children $v_1$ and $v_2$.
From \cref{lem phi proj}, every $\Phi \in \proj(v)$ can be constructed from one $\Phi_1 \in \proj(v_1)$ and one $\Phi_2 \in \proj(v_2)$ as in \eqref{cond l1}.
For each pair $\Phi_1, \Phi_2$, the construction requires $O\pare{\card{\cons}}$ time.
Since $T$ is of projection-width $\pwidth$, there are at most $\pwidth^2$ pairs, we can construct all maps in $\proj(v)$ in time $O\pare{\pwidth^2 \card{\cons}}$.
We then sort the obtained maps lexicographically in time $O\pare{\pwidth^2 \log(\pwidth) \card{\cons}}$, and then we delete duplicates in time $O\pare{\pwidth^2 \card{\cons}}$.
Therefore, for every inner vertex $v$ of $T$, the set $\proj(v)$ can be constructed in time $O\pare{\pwidth^2 \log(\pwidth) \card{\cons}}$.
Since there are $\card{\vars}+\card{\cons}-1$ inner vertices, they require $O\pare{\pwidth^2 \log(\pwidth) \pare{\card{\vars}+\card{\cons}} \card{\cons}}$ time.

\medskip

The construction of the sets $\proj(\overline{v})$, for every vertex $v$ of $T$, is performed in a top down manner.
For the root $r$ of $T$, the set $\proj(\overline{r})$ can be constructed, as in \cref{rem proj} \ref{rem proj root psi}, in time $O\pare{\card{\cons}}$.
Next, consider a vertex $v$ of $T$ with children $v_1$ and $v_2$.
From \cref{lem psi proj}, every $\Psi_1 \in \proj(\overline{v_1})$ can be constructed from one $\Psi \in \proj(\overline{v})$ and one $\Phi_2 \in \proj(v_2)$ as in \eqref{cond l2}.
For each pair $\Psi, \Phi_2$, the construction requires $O\pare{\card{\cons}}$ time.
Since $T$ is of projection-width $\pwidth$, there are at most $\pwidth^2$ pairs, and the total time required for $v_1$, including deleting duplicates, is $O\pare{\pwidth^2 \log(\pwidth) \card{\cons}}$.
Symmetrically, we can construct $\proj(\overline{v_2})$ using \eqref{cond l3}.
Since the total number of vertices of $T$ is $2\pare{\card{\vars}+\card{\cons}}-1$, the construction of the sets $\proj(\overline{v})$, for every vertex $v$ of $T$, requires $O\pare{\pwidth^2 \log(\pwidth) \pare{\card{\vars}+\card{\cons}} \card{\cons}}$ time.
\end{prf}

\section{Shapes}
\label{sec shapes}

Consider a separable system $\pare{\vars,\dom,\consin,\consge}$, let $\cons := \consin \cup \consge$, let $T$ be a branch decomposition of $\vars \cup \cons$, and let $v$ be a vertex of $T$.
A \emph{shape} (for $v$, with respect to $T$) is a pair
\begin{align*}
(\Phi,\Psi) 
\end{align*} 
such that $\Phi \in \proj(v)$ and $\Psi \in \proj(\overline v).$
In other words, $(\Phi,\Psi)$ is a shape if there exists $\ass : \vars_v \to \dom$ such that $\Phi = \overline{\cons_v} / \ass$, and there exists $\assb : \overline{\vars_v} \to \dom$ such that $\Psi = \cons_v / \assb$.
Observe that if $T$ is of projection-width $\pwidth$, then $\proj(v)$ and $\proj(\overline{v})$ have cardinality at most $\pwidth$, thus there are at most $\pwidth^2$ different shapes for $v$.

\paragraph{Linked shapes.}

We now define the concept of ``linked shapes'', which allows us to relate the shapes for an inner vertex $v$ to the shapes for its children $v_1$ and $v_2$.
Consider a separable system $\pare{\vars,\dom,\consin,\consge}$, let $\cons := \consin \cup \consge$, let $T$ be a branch decomposition of $\vars \cup \cons$, and let $v$ be an inner vertex of $T$ with children $v_1$ and $v_2$. 
We say that $(\Phi,\Psi)$, $(\Phi_1,\Psi_1)$, $(\Phi_2,\Psi_2)$ are \emph{linked shapes} for $v,v_1,v_2$, if they can be constructed as follows: 
First, let $\Psi \in \proj(\overline{v})$, $\Phi_1 \in \proj(v_1)$, and $\Phi_2 \in \proj(v_2)$ such that
\begin{align}
\label{cond l1 star}
\tag{L1$^*$}
\Phi_1^\con + \Phi_2^\con & \le \gamma^\con
&& \forall \con \in \overline{\cons_v} \cap \consin, \\
\label{cond l2 star}
\tag{L2$^*$}
\Psi^\con + \Phi_2^\con & \le \gamma^\con
&& \forall \con \in \cons_{v_1} \cap \consin, \\
\label{cond l3 star}
\tag{L3$^*$}
\Psi^\con + \Phi_1^\con & \le \gamma^\con
&& \forall \con \in \cons_{v_2} \cap \consin.
\end{align}
Then define $\Phi$, $\Psi_1$, and $\Psi_2$ according to \eqref{cond l1}, \eqref{cond l2}, and \eqref{cond l3}, respectively.
We also say that the above linked shapes \emph{originated from} $\Psi$, $\Phi_1$, $\Phi_2$.
Since $\Phi^\con, \Psi_1^\con, \Psi_2^\con,  \in \bra{0,1,\dots,\gamma^\con}$, \eqref{cond l1}, \eqref{cond l1 star}, \eqref{cond l2}, \eqref{cond l2 star}, \eqref{cond l3}, \eqref{cond l3 star} are equivalent to
\begin{align}
\label{cond l1 in}
\tag{L1$^\in$}
\Phi^\con & = \Phi_1^\con + \Phi_2^\con 
&& \forall \con \in \overline{\cons_v} \cap \consin, \\
\label{cond l2 in}
\tag{L2$^\in$}
\Psi_1^\con & = \Psi^\con + \Phi_2^\con
&& \forall \con \in \cons_{v_1} \cap \consin, \\
\label{cond l3 in}
\tag{L3$^\in$}
\Psi_2^\con & = \Psi^\con + \Phi_1^\con
&& \forall \con \in \cons_{v_2} \cap \consin, \\
\label{cond l1 ge}
\tag{L1$^\ge$}
\Phi^\con & = \min\bra{\Phi_1^\con + \Phi_2^\con, \gamma^\con} && \forall \con \in \overline{\cons_v} \cap \consge, \\
\label{cond l2 ge}
\tag{L2$^\ge$}
\Psi_1^\con & = \min\bra{\Psi^\con + \Phi_2^\con, \gamma^\con} && \forall \con \in \cons_{v_1} \cap \consge, \\
\label{cond l3 ge}
\tag{L3$^\ge$}
\Psi_2^\con & = \min\bra{\Psi^\con + \Phi_1^\con, \gamma^\con} && \forall \con \in \cons_{v_2} \cap \consge.
\end{align}
Note that linked shapes are indeed shapes.
In fact, \cref{lem phi proj} implies that $\Phi \in \proj(v)$, thus $(\Phi,\Psi)$ is a shape for $v$.
On the other hand, \cref{lem psi proj} implies $\Psi_1 \in \proj(\overline{v_1})$ and $\Psi_2 \in \proj(\overline{v_2})$, so $(\Phi_1,\Psi_1)$ is a shape for $v_1$ and $(\Phi_2,\Psi_2)$ is a shape for $v_2$.

\paragraph{Assignments of shape.}

Consider a separable system $\pare{\vars,\dom,\consin,\consge}$, let $\cons := \consin \cup \consge$, let $T$ be a branch decomposition of $\vars \cup \cons$, and let $v$ be a vertex of $T$.
Given a shape $(\Phi,\Psi)$ for $v$, we say that $\ass : \vars_v \to \dom$ has \emph{shape} $(\Phi,\Psi)$ if 
\begin{align}
\tag{S1}
\label{cond s1}
& \Phi^\con = \min\bra{\cb^\con\pare{\ass},\gamma^\con} && \forall \con \in \overline{\cons_v}, \\
\tag{S1$^*$}
\label{cond s1 c}
& \cb^\con\pare{\ass} \le \gamma^\con && \forall \con \in \overline{\cons_v} \cap \consin, \\
\tag{S2$^\in$}
\label{cond s2 in}
& \Psi^\con + \cb^\con\pare{\ass} \in \Gamma^\con && \forall \con \in \cons_v \cap \consin, \\
\tag{S2$^\ge$}
\label{cond s2 ge}
& \Psi^\con + \cb^\con\pare{\ass} \ge \gamma^\con && \forall \con \in \cons_v \cap \consge.
\end{align} 
Note that \eqref{cond s1} can be written in the form $\Phi = \overline{\cons_v} / \ass$.
Also, since $\Phi^\con \le \gamma^\con$, \eqref{cond s1} and \eqref{cond s1 c} are equivalent to
\begin{align}
\tag{S1$^\in$}
\label{cond s1 in}
& \Phi^\con = \cb^\con\pare{\ass} && \forall \con \in \overline{\cons_v} \cap \consin, \\
\tag{S1$^\ge$}
\label{cond s1 ge}
& \Phi^\con = \min\bra{\cb^\con\pare{\ass},\gamma^\con} && \forall \con \in \overline{\cons_v} \cap \consge.
\end{align}
Note that an assignment $\ass : \vars_v \to \dom$ can have more than one shape. 
If $\ass$ has shape $(\Phi_1,\Psi_1)$ and shape $(\Phi_2,\Psi_2)$, then it only implies $\Phi_1=\Phi_2=\overline{\cons_v} / \ass$.

The intuition behind the notion of ``assignment of shape'' is that if $\ass : \vars_v \to \dom$ has shape $(\Phi,\Psi)$, then it can be extended to an assignment satisfying all constraints in $\cons_v$ by combining it with $\assb : \overline{\vars_v} \to \dom$ such that $\cons_v / \assb \ge \Psi$.
The next result details how the notion of shape allows us to obtain the assignments that satisfy all the constraints in $\cons$.

\begin{remark}
\label{rem root vshape}
Consider a separable system $\pare{\vars,\dom,\consin,\consge}$, let $\cons := \consin \cup \consge$, let $T$ be a branch decomposition of $\vars \cup \cons$, and let $r$ be the root of $T$.
Consider the shape $(\epsilon, \cons / \epsilon)$ for $r$ (see \cref{rem proj} \ref{rem proj root phi} and \ref{rem proj root psi}), and
note that $\vars_r = \vars$, $\cons_r = \cons$, and $\overline{\cons_r} = \emptyset$.
Therefore, \eqref{cond s1 in}, \eqref{cond s1 ge}, \eqref{cond s2 in}, \eqref{cond s2 ge} imply that an assignment $\ass : \vars \to \dom$ has shape $(\epsilon, \cons / \epsilon)$ if and only if 
\begin{align*}
& \cb^\con\pare{\ass} \in \Gamma^\con && \forall \con \in \consin, \\
& \cb^\con\pare{\ass} \ge \gamma^\con && \forall \con \in \consge.
\end{align*}
Since
\begin{align*}
\cb^\con\pare{\ass}
= \sum_{\var \in \vars} g^\con_\var(\ass(\var)),
\end{align*}
this happens if and only if $\ass$ satisfies the constraints in $\cons$.
\end{remark}

In particular, \cref{rem root vshape} implies that there can be shapes $(\Phi,\Psi)$ for $v$ such that there is no $\ass : \vars_v \to \dom$ of shape $(\Phi,\Psi)$.

\subsection{From children to parent}

Our next goal is to relate the ``assignments of shape'' for an inner vertex $v$ to the ``assignments of shape'' for its children $v_1$ and $v_2$. 
First, we study this connection when traversing the tree in a bottom up manner.

\begin{lemma}
\label{lem up}
Consider a separable system $\pare{\vars,\dom,\consin,\consge}$, let $\cons := \consin \cup \consge$, let $T$ be a branch decomposition of $\vars \cup \cons$, and let $v$ be an inner vertex of $T$ with children $v_1$ and $v_2$. 
Let $(\Phi,\Psi)$, $(\Phi_1,\Psi_1)$, $(\Phi_2,\Psi_2)$ be linked shapes for $v,v_1,v_2$.
If $\ass_1 : \vars_{v_1} \to \dom$ has shape $(\Phi_1,\Psi_1)$ and $\ass_2 : \vars_{v_2} \to \dom$ has shape $(\Phi_2,\Psi_2)$, then $\ass := \ass_1 \cup \ass_2 : \vars_v \to \dom$ has shape $(\Phi,\Psi)$.
\end{lemma}

\begin{prf}
Let $\ass_1 : \vars_{v_1} \to \dom$ of shape $(\Phi_1,\Psi_1)$ and $\ass_2 : \vars_{v_2} \to \dom$ of shape $(\Phi_2,\Psi_2)$.
To prove that $\ass$ has shape $(\Phi,\Psi)$, we need to show that $\ass$ satisfies conditions~\eqref{cond s1 in}, \eqref{cond s1 ge}, \eqref{cond s2 in}, \eqref{cond s2 ge}.

\textbf{Condition~\eqref{cond s1 in}.} 
Let $\con \in \overline{\cons_v} \cap \consin$, and note that $j \in \overline{\cons_{v_1}} \cap \overline{\cons_{v_2}}$.
We have
\begin{align*}
\Phi^\con
& =
\Phi_1^\con + \Phi_2^\con
&& \text{(from \eqref{cond l1 in})} \\
& =
\cb^\con\pare{\ass_1} + \cb^\con\pare{\ass_2}
&& \text{(from \eqref{cond s1 in} for $\ass_1$ and $\ass_2$)} \\
& =
\cb^\con\pare{\ass}
&& \text{(from \cref{obs vb} \ref{obs vb 3})}.
\end{align*}
Therefore, $\ass$ satisfies condition~\eqref{cond s1 in}.

\textbf{Condition~\eqref{cond s1 ge}.}
Let $\con \in \overline{\cons_v} \cap \consge$, and note that $j \in \overline{\cons_{v_1}} \cap \overline{\cons_{v_2}}$.
We have
\begin{align*}
\Phi^\con
& =
\min\bra{\Phi_1^\con + \Phi_2^\con, \gamma^\con}
&& \text{(from \eqref{cond l1 ge})} \\
& =
\min\bra{\min\bra{\cb^\con\pare{\ass_1},\gamma^\con} + \min\bra{\cb^\con\pare{\ass_2},\gamma^\con}, \gamma^\con}
&& \text{(from \eqref{cond s1 ge} for $\ass_1$ and $\ass_2$)} \\
& =
\min\bra{\cb^\con\pare{\ass},\gamma^\con}
&& \text{(from \cref{lem phi ass})}.
\end{align*}
Therefore, $\ass$ satisfies condition~\eqref{cond s1 ge}.

\textbf{Condition~\eqref{cond s2 in}.} 
Let $\con \in \cons_v \cap \consin$.
Then $\con \in \cons_{v_1} \cap \overline{\cons_{v_2}}$ or $\con \in \cons_{v_2} \cap \overline{\cons_{v_1}}$.
We assume, without loss of generality, that $\con \in \cons_{v_1} \cap \overline{\cons_{v_2}}$, since the other case is symmetric.
We get
\begin{align*}
\Psi^\con + \cb^\con\pare{\ass} & =
\Psi^\con + \cb^\con\pare{\ass_1} + \cb^\con\pare{\ass_2}
&& \text{(from \cref{obs vb} \ref{obs vb 3})} \\
& = \Psi^\con + \cb^\con\pare{\ass_1} + \Phi_2^\con
&& \text{(from \eqref{cond s1 in} for $\ass_2$)} \\
& = \Psi_1^\con + \cb^\con\pare{\ass_1}
&& \text{(from \eqref{cond l2 in})} \\
& \in \Gamma^\con
&& \text{(from \eqref{cond s2 in} for $\ass_1$).}
\end{align*}
Therefore, $\ass$ satisfies condition~\eqref{cond s2 in}.

\textbf{Condition~\eqref{cond s2 ge}.} 
Let $\con \in \cons_v \cap \consge$.
Then $\con \in \cons_{v_1} \cap \overline{\cons_{v_2}}$ or $\con \in \cons_{v_2} \cap \overline{\cons_{v_1}}$.
We assume, without loss of generality, that $\con \in \cons_{v_1} \cap \overline{\cons_{v_2}}$, since the other case is symmetric.
We get 
\begin{align*}
\Psi^\con + \cb^\con\pare{\ass} & = \Psi^\con + \cb^\con\pare{\ass_1} + \cb^\con\pare{\ass_2} && \text{(from \cref{obs vb} \ref{obs vb 3})} \\
& \ge \Psi^\con + \cb^\con\pare{\ass_1} + \min\bra{\cb^\con\pare{\ass_2},\gamma^\con} \\
& = \Psi^\con + \cb^\con\pare{\ass_1} + \Phi_2^\con && \text{(from \eqref{cond s1 ge} for $\ass_2$)} \\
& \ge \min\bra{\Psi^\con + \Phi_2^\con, \gamma^\con} + \cb^\con\pare{\ass_1} \\
& = \Psi_1^\con + \cb^\con\pare{\ass_1} && \text{(from \eqref{cond l2 ge})} \\
& \ge \gamma^\con && \text{(from \eqref{cond s2 ge} for $\ass_1$)}.
\end{align*}
Therefore, $\ass$ satisfies condition~\eqref{cond s2 ge}.
\end{prf}

\subsection{From parent to children}

Next, we study this connection among ``assignments of shape'' when traversing the tree in a top down manner.

\begin{lemma}
\label{lem down}
Consider a separable system $\pare{\vars,\dom,\consin,\consge}$, let $\cons := \consin \cup \consge$, let $T$ be a branch decomposition of $\vars \cup \cons$, and let $v$ be an inner vertex of $T$ with children $v_1$ and $v_2$. 
Let $(\Phi,\Psi)$ be a shape for $v$, and let $\ass : \vars_v \to \dom$ of shape $(\Phi,\Psi)$. 
Let $\ass_1$ and $\ass_2$ denote the restrictions of $\ass$ to $\vars_{v_1}$ and $\vars_{v_2}$, respectively. 
There is a unique pair of shapes $(\Phi_1,\Psi_1)$ for $v_1$ and $(\Phi_2,\Psi_2)$ for $v_2$ such that $(\Phi,\Psi)$, $(\Phi_1,\Psi_1)$, $(\Phi_2,\Psi_2)$ are linked shapes, $\ass_1$ has shape $(\Phi_1,\Psi_1)$, and $\ass_2$ has shape $(\Phi_2,\Psi_2)$.
\end{lemma}

\begin{prf}
Since $\ass_1 : \vars_{v_1} \to \dom$ must have shape $(\Phi_1,\Psi_1)$, and $\ass_2 : \vars_{v_2} \to \dom$ must have shape $(\Phi_2,\Psi_2)$, according to \eqref{cond s1} we need to set $\Phi_1 := \overline{\cons_{v_1}} / \ass_1 \in \proj(v_1)$ and $\Phi_2 := \overline{\cons_{v_2}} / \ass_2 \in \proj(v_2)$.
Then, \cref{lem phi ass} implies that \eqref{cond l1} holds.
Therefore, the only triple of linked shapes that we can consider is the one originated from $\Psi$, $\Phi_1$, $\Phi_2$, which we denote by $(\Phi,\Psi)$, $(\Phi_1,\Psi_1)$, $(\Phi_2,\Psi_2)$.
Note that, according to \cref{lem phi ass} and \eqref{cond l1}, the linked shape $(\Phi,\Psi)$ that we  obtained is indeed the shape for $v$ we started from.

To complete the proof, we only need to show that $\ass_1$ has shape $(\Phi_1,\Psi_1)$ and $\ass_2$ has shape $(\Phi_2,\Psi_2)$.
We only show it for $\ass_1$, as the other one is symmetric.
We already know that \eqref{cond s1} holds for $\ass_1$, thus to prove that $\ass_1$ has shape $(\Phi_1,\Psi_1)$, it suffices to show that $\ass_1$ satisfies~\eqref{cond s1 c}, \eqref{cond s2 in}, \eqref{cond s2 ge}.

\textbf{Condition~\eqref{cond s1 c}.}
Let $j \in \overline{\cons_{v_1}} \cap \consin$.
We consider separately two cases.
If $\con \in \overline{\cons_v}$, we have
\begin{align*}
\cb^\con\pare{\ass_1}
& = 
\cb^\con\pare{\ass} - \cb^\con\pare{\ass_2}
&& \text{(from \cref{obs vb} \ref{obs vb 3})} \\
& \le
\cb^\con\pare{\ass}
&& \text{(from \cref{obs vb} \ref{obs vb 1})} \\
& \le
\gamma^\con
&& \text{(from \eqref{cond s1 c} for $\ass$)}.
\end{align*}
If $\con \in \cons_v$, we have
\begin{align*}
\cb^\con\pare{\ass_1}
& = 
\cb^\con\pare{\ass} - \cb^\con\pare{\ass_2}
&& \text{(from \cref{obs vb} \ref{obs vb 3})} \\
& \le
\cb^\con\pare{\ass}
&& \text{(from \cref{obs vb} \ref{obs vb 1})} \\
& \in
\Gamma^\con - \Psi^\con
&& \text{(from \eqref{cond s2 in} for $\ass$)} \\
& \le
\gamma^\con.
\end{align*}
Therefore, $\ass$ satisfies condition~\eqref{cond s1 c}.

\textbf{Condition~\eqref{cond s2 in}.} 
Let $\con \in \cons_{v_1} \cap \consin$.
Then $\con \in \cons_{v} \cap \overline{\cons_{v_2}}$.
We have
\begin{align*}
\Psi_1^\con + \cb^\con\pare{\ass_1}
& = \min\bra{\Psi^\con + \Phi_2^\con, \gamma^\con} + \cb^\con\pare{\ass_1}
&& \text{(from \eqref{cond l2})} \\
& = \min\bra{\Psi^\con + \cb^\con\pare{\ass_2}, \gamma^\con} + \cb^\con\pare{\ass_1}
&& \text{(from \eqref{cond s1 in} for $\ass_2$)} \\
& = \min\bra{\Psi^\con + \cb^\con\pare{\ass_1} + \cb^\con\pare{\ass_2}, \gamma^\con + \cb^\con\pare{\ass_1}} \\
& = \min\bra{\Psi^\con + \cb^\con\pare{\ass}, \gamma^\con + \cb^\con\pare{\ass_1}}
&& \text{(from \cref{obs vb} \ref{obs vb 3})} \\
& = \Psi^\con + \cb^\con\pare{\ass}
&& \text{(from \eqref{cond s2 in} for $\ass$)} \\
& \in \Gamma^\con
&& \text{(from \eqref{cond s2 in} for $\ass$).} \\
\end{align*}
Therefore, $\ass$ satisfies condition~\eqref{cond s2 in}.

\textbf{Condition~\eqref{cond s2 ge}.} 
Let $\con \in \cons_{v_1} \cap \consge$.
Then $\con \in \cons_{v} \cap \overline{\cons_{v_2}}$.
We have
\begin{align*}
\Psi_1^\con + \cb^\con\pare{\ass_1}
& = \min\bra{\Psi^\con + \Phi_2^\con, \gamma^\con} + \cb^\con\pare{\ass_1}
&& \text{(from \eqref{cond l2})} \\
& = \min\bra{\Psi^\con + \min\bra{\cb^\con\pare{\ass_2},\gamma^\con}, \gamma^\con} + \cb^\con\pare{\ass_1}
&& \text{(from \eqref{cond s1 ge} for $\ass_2$)} \\
& = \min\bra{\Psi^\con + \cb^\con\pare{\ass_2}, \gamma^\con} + \cb^\con\pare{\ass_1}
&& \text{(from \cref{lem 3min})} \\
& = \min\bra{\Psi^\con + \cb^\con\pare{\ass_1} + \cb^\con\pare{\ass_2}, \cb^\con\pare{\ass_1} + \gamma^\con} \\
& = \min\bra{\Psi^\con + \cb^\con\pare{\ass}, \cb^\con\pare{\ass_1} + \gamma^\con}
&& \text{(from \cref{obs vb} \ref{obs vb 3})} \\
& \ge \gamma^\con
&& \text{(from \eqref{cond s2 ge} for $\ass$)}.
\end{align*}
Therefore, $\ass$ satisfies condition~\eqref{cond s2 ge}.
\end{prf}

\subsection{Structure of shapes}

The next key result is a direct consequence of \cref{lem up,lem down}, and will play a major role in our main algorithms.
To state it, we define 
\begin{align*}
\Ass_v(\Phi,\Psi) 
:= \bra{\ass : \vars_v \to \dom \mid \ass \text{ has shape } (\Phi,\Psi)}.
\end{align*}

\begin{proposition}
\label{prop decomposition}
Consider a separable system $\pare{\vars,\dom,\consin,\consge}$, let $\cons := \consin \cup \consge$, let $T$ be a branch decomposition of $\vars \cup \cons$, and let $v$ be an inner vertex of $T$ with children $v_1$ and $v_2$. 
Let $(\Phi,\Psi)$ be a shape for $v$, and let $\pairs$ denote the set of pairs of shapes $(\Phi_1,\Psi_1)$ for $v_1$ and $(\Phi_2,\Psi_2)$ for $v_2$ such that $(\Phi,\Psi)$, $(\Phi_1,\Psi_1)$, $(\Phi_2,\Psi_2)$ are linked shapes. 
Then $\Ass_v(\Phi,\Psi)$ is the union of disjoint sets
\begin{align*}
\bra{\ass_1 \cup \ass_2 \mid \ass_1 \in \Ass_{v_1}(\Phi_1,\Psi_1), \ass_2 \in \Ass_{v_2}(\Phi_2,\Psi_2)}, \qquad \forall \pare{(\Phi_1,\Psi_1), (\Phi_2,\Psi_2)} \in \pairs.
\end{align*}
\end{proposition}

\begin{prf}
Containment $\supseteq$ follows from \cref{lem up}, while containment $\subseteq$ follows from \cref{lem down}.
The fact that the union is disjoint follows from the uniqueness in \cref{lem down}.
\end{prf}

\section{Algorithms}
\label{sec algorithms}

\subsection{Optimization}
\label{sec opt}

In this section, we consider our first problem defined over a separable system, and we show how shapes can be used to solve it.
In the \emph{optimization problem,} we are given a separable system $\pare{\vars,\dom,\consin,\consge}$ and $\nu_\var : \dom \to \R$ for every $\var \in \vars$.
For every assignment $\ass : \vars \to \dom$, we define its \emph{value}
\begin{align}
\label{eq ass value}
\val(\ass) & := \sum_{\var \in \vars} \nu_\var \pare{\ass(\var)}.
\end{align}
The goal is to find a highest-value assignment from $\vars$ to $\dom$ satisfying the constraints in $\consin \cup \consge$, or prove that no such assignment exists.
In the next result, we show that we can solve efficiently the optimization problem on separable systems with bounded projection-width.

\begin{theorem}[Optimization]
\label{th optimization}
Consider a separable system $\pare{\vars,\dom,\consin,\consge}$, let $\cons := \consin \cup \consge$, and let $T$ be a branch decomposition of $\vars \cup \cons$ of projection-width $\pwidth$.
Let $\nu_\var : \dom \to \R$ for every $\var \in \vars$.
Then, the optimization problem can be solved in time 
\begin{equation}
\label{eq opt runtime}
O\pare{
\pwidth^3 \pare{\card{\vars}+\card{\cons}} \card{\cons}
+
\pwidth \card{\consin} \Lambda
+
\card{\vars} \card{\cons} \card{\dom}\log\pare{\card{\dom}}}.
\end{equation}
\end{theorem}


\begin{prf}
First, we apply \cref{prop all proj} and compute $\proj(v)$ and $\proj(\overline{v})$, for every vertex $v$ of $T$, in time $O\pare{\pwidth^2 \log(\pwidth) \pare{\card{\vars}+\card{\cons}} \card{\cons} + \card{\vars}\card{\cons}\card{\dom}\log\pare{\card{\dom}}}$.

\textbf{Table.}
Next, our algorithm will construct, for each vertex $v$ of $T$, a table $\tab_v$ indexed by the shapes $(\Phi,\Psi)$ for $v$.
For every assignment $\ass : \vars_v \to \dom$, we define its \emph{value}
\begin{align*}
\val_{v}(\ass) := \sum_{\var \in \vars_v} \nu_\var \pare{\ass(\var)}.
\end{align*}
For a shape $(\Phi,\Psi)$, the content of the table $\tab_v$ at this index, which we denote by $\tab_v(\Phi,\Psi)$, should be a pair $(\ass,\val_{v}(\ass))$, where $\ass$ is a highest-value assignment from $\vars_v$ to $\dom$ of shape $(\Phi,\Psi)$.
If no such assignment exists, we should have $\tab_v(\Phi,\Psi) = \NA$.
We now explain how we can compute the table $\tab_v$, for every vertex $v$ of $T$.
This is done in a bottom up manner.

\textbf{Leaves corresponding to constraints.}
Consider a leaf $l$ of $T$ corresponding to a constraint $\con \in \cons$.
From \cref{rem proj} \ref{rem proj constraint-leaf}, there is only one assignment from $\vars_l$ to $\dom$, the empty assignment $\epsilon$, which has value $0$, and that can be constructed in time $O(1)$.
We now find all shapes $(\Phi,\Psi)$ for $l$ such that $\epsilon$ has shape $(\Phi,\Psi)$; We then set $\tab_l(\Phi,\Psi) := (\epsilon,0)$ if $\epsilon$ has shape $(\Phi,\Psi)$, and $\tab_l(\Phi,\Psi) := \NA$ otherwise.
Clearly, there is only one $\Phi \in \proj(l)$, and it satisfies \eqref{cond s1 in} and \eqref{cond s1 ge}.
For every $\Phi \in \proj(\overline{l})$, we need to check \eqref{cond s2 in} or \eqref{cond s2 ge}, depending on whether $c$ is in $\consin$ or $\consge$, and this can be done in time $O\pare{\Lambda}$ and $O(1)$ respectively.
Therefore, we can compute the table $\tab_l$ in time $O\pare{\pwidth \Lambda}$ if $c \in \consin$, and in time $O\pare{\pwidth}$ if $c \in \consge$.
Since there are $\card{\consin}$ leaves corresponding to constraints in $\consin$, and $\card{\consge}$ leaves corresponding to constraints in $\consge$, in total they require $O\pare{\pwidth \card{\consin} \Lambda + \pwidth \card{\consge}}$ time.

\textbf{Leaves corresponding to variables.}
Consider a leaf $l$ of $T$ corresponding to a variable $\var \in \vars$.
From \cref{rem proj} \ref{rem proj variable-leaf}, there are $\card{\dom}$ assignment from $\vars_l$ to $\dom$, and they can be constructed, with their values, in time $O(\card{\dom})$.
We now fix one such assignment $\ass$, and find all shapes $(\Phi,\Psi)$ for $l$ such that $\ass$ has shape $(\Phi,\Psi)$.
There is at most one $\Phi \in \proj(l)$ satisfying \eqref{cond s1 in} and \eqref{cond s1 ge}, and it can be constructed in time $O(\cons)$.
We fix such $\Phi$, and observe that $\ass$ has shape $(\Phi,\Psi)$, for every $\Psi \in \proj(\overline{l})$, which are at most $\pwidth$.
This is because
conditions \eqref{cond s2 in}, \eqref{cond s2 ge} are
always satisfied since $\cons_v = \emptyset$.
Once we have considered all $\card{\dom}$ assignments, for every shape $(\Phi,\Psi)$ for $l$, which are at most $\pwidth^2$, we set $\tab_l(\Phi,\Psi) := (\ass,\val_{v}(\ass))$, where $\ass$ is the highest-value assignment of shape $(\Phi,\Psi)$, or $\tab_l(\Phi,\Psi) := \NA$ if no assignment has shape $(\Phi,\Psi)$. 
Therefore, we can compute the table $\tab_l$ in time $O\pare{\card{\cons} \card{\dom} + \pwidth^2}$.
Since there are $\card{\vars}$ leaves corresponding to variables, in total they require $O\pare{\card{\vars} \card{\cons} \card{\dom} + \pwidth^2 \card{\vars}}$ time.

\textbf{Inner vertices.}
Consider now an inner vertex $v$ of $T$, with children $v_1,v_2$.
We loop over all triples $(\Phi,\Psi)$, $(\Phi_1,\Psi_1)$, $(\Phi_2,\Psi_2)$ of linked shapes for $v,v_1,v_2$. 
To do so, we pick $\Psi \in \proj(\overline{v})$, $\Phi_1 \in \proj(v_1)$, and $\Phi_2 \in \proj(v_2)$, we check \eqref{cond l1 star}, \eqref{cond l2 star}, \eqref{cond l1 star}, and if they are satisfied, we define $\Phi$, $\Psi_1$, and $\Psi_2$ according to \eqref{cond l1}, \eqref{cond l2}, and \eqref{cond l3}, respectively.
Note that we have at most $\pwidth^3$ such triples, and for each the above check and construction requires $O(\card{\cons})$ time.
For each triple of linked shapes with $\tab_{v_1}(\Phi_1,\Psi_1) \neq \NA$ and $\tab_{v_2}(\Phi_2,\Psi_2) \neq \NA$, let $\pare{\ass_1,\val_{v_1}(\ass_1)} := \tab_{v_1}(\Phi_1,\Psi_1)$ and $\pare{\ass_2,\val_{v_2}(\ass_2)} := \tab_{v_2}(\Phi_2,\Psi_2)$.
We then construct the value of the ``candidate assignment for $(\Phi,\Psi)$'' given by $\ass_1 \cup \ass_2$ in time $O(1)$ by summing $\val_{v_1}(\ass_1)$ and $\val_{v_2}(\ass_2)$.

For each shape $(\Phi,\Psi)$ for $v$, we then set the highest-value assignment $\ass_1 \cup \ass_2$, among all candidate assignments for $(\Phi,\Psi)$, as the content of $\tab_v(\Phi,\Psi)$, together with its value.
If there are no candidate assignments for $(\Phi,\Psi)$, we set $\tab_l(\Phi,\Psi) := \NA$.
To improve runtime, here we do not explicitly construct $\ass_1 \cup \ass_2$; instead, we store pointers to $(\Phi_1,\Psi_1)$ and $(\Phi_2,\Psi_2)$ giving the highest-value, so that this is done in time $O(1)$ instead of $O\pare{\card{\vars}}$.
\cref{prop decomposition} implies that we set $\tab_v$ correctly, for every inner vertex $v$ of $T$.
Therefore, for each inner vertex $v$ of $T$, the table $\tab_v$ can be computed (partially implicitly) in time $O(\pwidth^3 \card{\cons})$.
Since there are $\card{\vars}+\card{\cons}-1$ inner vertices of $T$, in total they require $O(\pwidth^3 \pare{\card{\vars}+\card{\cons}} \card{\cons})$ time.

\textbf{Root.}
Once the table $\tab_v$ is computed, for every vertex $v$ of $T$, we consider the root $r$ of $T$, and the shape $\pare{\epsilon, \cons/\epsilon}$ for $r$ (see \cref{rem proj} \ref{rem proj root phi} and \ref{rem proj root psi}).
From \cref{rem root vshape}, 
the assignments $\ass : \vars \to \dom$ that have this shape are precisely those that satisfy the constraints in $\cons$.
The table $\tab_r$, indexed by $(\epsilon, \cons / \epsilon)$, implicitly contains a highest-value assignment from $\vars$ to $\dom$ satisfying the constraints in $C$.
Following our pointers, we can construct it explicitly in time $O\pare{\card{\vars}+\card{\cons}}$.
\end{prf}

\subsection{Counting}
\label{sec counting}

In this section, we consider the \emph{counting problem.} 
In this problem, we are given a separable system $\pare{\vars,\dom,\consin,\consge}$, and our goal is to return the number of assignments $\ass : \vars \to \dom$ satisfying the constraints in $\consin \cup \consge$.
Next, in \cref{th counting}, we show how we can solve efficiently the counting problem on separable systems with bounded projection-width.

\begin{theorem}[Counting]
\label{th counting}
Consider a separable system $\pare{\vars,\dom,\consin,\consge}$, let $\cons := \consin \cup \consge$, and let $T$ be a branch decomposition of $\vars \cup \cons$ of projection-width $\pwidth$.
Then, the counting problem can be solved in time \eqref{eq opt runtime}.
\end{theorem}

\begin{prf}
The proof follows the same structure as that of \cref{th optimization}, and we only highlight the differences here.

\textbf{Table.}
For a shape $(\Phi,\Psi)$, the content of the table $\tab_v$ at this index
should be the number of assignments $\ass : \vars_v \to \dom$ of shape $(\Phi,\Psi)$.

\textbf{Leaves corresponding to constraints.}
Here, we set $\tab_l(\Phi,\Psi) := 1$ if $\epsilon$ has shape $(\Phi,\Psi)$, and $\tab_l(\Phi,\Psi) := 0$ otherwise.

\textbf{Leaves corresponding to variables.}
Here, we set $\tab_l(\Phi,\Psi)$ to be the number of assignments that have shape $(\Phi,\Psi)$, among the $\card{\dom}$ that we constructed.

\textbf{Inner vertices.}
Here, we first initialize $\tab_v(\Phi,\Psi) := 0$ for every shape $(\Phi,\Psi)$ for $v$.
Then, for each triple $(\Phi,\Psi)$, $(\Phi_1,\Psi_1)$, $(\Phi_2,\Psi_2)$ of linked shapes for $v,v_1,v_2$, we let $\val_1 := \tab_{v_1}(\Phi_1,\Psi_1)$ and $\val_2 := \tab_{v_2}(\Phi_2,\Psi_2)$, and add $\val_1 \val_2$ to $\tab_v(\Phi,\Psi)$ in time $O(1)$.

\textbf{Root.}
From \cref{rem root vshape}, 
the content of $\tab_r$, indexed by $(\epsilon, \cons / \epsilon)$, is the number of assignments $\ass : \vars \to \dom$ that satisfy the constraints in $\cons$.
\end{prf}

\subsection{Top-$k$}
\label{sec top-k}

In this section, we consider the \emph{top-$k$ problem.} 
In this problem, we are given a separable system $\pare{\vars,\dom,\consin,\consge}$ and $\nu_\var : \dom \to \R$ for every $\var \in \vars$.
The value of an assignment is defined by \eqref{eq ass value}, like for the optimization problem.
The goal is to return a sorted list of $k$ highest-value assignments from $\vars$ to $\dom$ that satisfy the constraints in $\cons$.
More formally, the output should be a list of $k$ assignments $\ass_1,\ass_2,\dots,\ass_k$ from $\vars$ to $\dom$ satisfying the constraints in $\cons$ and such that
\begin{equation*}
\begin{aligned}
& \val(\ass_k) \le \val(\ass_{k-1}) \le \cdots \le \val\pare{\ass_1}, \\
& \val(\assb) \le \val(\ass_k) \text{ for every other $\assb : \vars \to \dom$ satisfying the constraints in $\cons$},
\end{aligned}
\end{equation*}
with the understanding that, in case there are only $h < k$ assignments from $\vars$ to $\dom$ satisfying the constraints in $\cons$, the output should be a sorted list of only those $h$ assignments.
Next, in \cref{th top-k}, we show how we can solve efficiently the top-$k$ problem on separable systems with bounded projection-width.

\begin{theorem}[Top-$k$]
\label{th top-k}
Consider a separable system $\pare{\vars,\dom,\consin,\consge}$, let $\cons := \consin \cup \consge$, and let $T$ be a branch decomposition of $\vars \cup \cons$ of projection-width $\pwidth$.
Let $\nu_\var : \dom \to \R$ for every $\var \in \vars$.
Then, the top-$k$ problem can be solved in time 
\begin{equation*}
O\pare{
\pwidth^3 \pare{\card{\vars}+\card{\cons}} \pare{\card{\cons} + k \log(k)}
+
\pwidth \card{\consin} \Lambda
+
\card{\vars} \card{\cons} \card{\dom}\log\pare{\card{\dom}}}.
\end{equation*}
\end{theorem}

\begin{prf}
The proof follows the same structure as that of \cref{th optimization}, and we only highlight the differences here.

\textbf{Table.}
For a shape $(\Phi,\Psi)$, the content of the table $\tab_v$ at this index
should be a sorted list of $k$ highest-value assignments from $\vars_v$ to $\dom$ of shape $(\Phi,\Psi)$, with their respective values.

\textbf{Leaves corresponding to constraints.}
Here, we set $\tab_l(\Phi,\Psi) := (\epsilon,0)$ if $\epsilon$ has shape $(\Phi,\Psi)$, and we set $\tab_l(\Phi,\Psi)$ as an empty list otherwise.

\textbf{Leaves corresponding to variables.}
Here, we set $\tab_l(\Phi,\Psi)$ to be a sorted list of $k$-highest value assignments of shape $(\Phi,\Psi)$, with their respective values, among the $\card{\dom}$ that we constructed.
The only extra step required, after constructing the $\card{\dom}$ assignments, is to order them according to their value, which can be done in time $O(\card{\dom}\log\pare{\card{\dom}})$.

\textbf{Inner vertices.}
Fix a triple $(\Phi,\Psi)$, $(\Phi_1,\Psi_1)$, $(\Phi_2,\Psi_2)$ of linked shapes for $v,v_1,v_2$.
It is well known that we can find $k$ largest values, in sorted order, in the Cartesian sum of two sorted arrays in time $O(k \log(k))$, using a best-first search strategy with a max-heap, and that 
only $k$ largest elements in each array need to be considered.
We then construct the values of the $k$ highest-value assignments, in sorted order, among all assignments 
of the form $\ass_1 \cup \ass_2$ with $\ass_1$ of shape $(\Phi_1,\Psi_1)$ and $\ass_2$ of shape $(\Phi_2,\Psi_2)$, by only considering those with $\ass_1$ in $\tab_{v_1}(\Phi_1,\Psi_1)$ and $\ass_2$ in $\tab_{v_2}(\Phi_2,\Psi_2)$.
We call these $k$ highest-value assignments a ``candidate top-$k$ for $(\Phi,\Psi)$.''
Since we store the corresponding assignments implicitly, the total time for one triple is $O(k \log(k))$.
This is done for each triple of linked shapes.

Now fix one shape $(\Phi,\Psi)$ for $v$, consider all candidate top-$k$ for $(\Phi,\Psi)$, and denote by $N_{(\Phi,\Psi)} \le \pwidth^3$ their number.
It then follows from \cref{prop decomposition} that we can set the content of $\tab_v(\Phi,\Psi)$ by finding the $k$ highest-value assignments, in sorted order, among all $N_{(\Phi,\Psi)}$ candidate top-$k$ for $(\Phi,\Psi)$.
It is well known that we can find $k$ largest values, in sorted order, in $N$ sorted arrays of $k$ elements each in time $O(N+k \log (N))$, using a $k$-way merge with a max-heap.
Since we store the assignments implicitly, the total time for the merge corresponding to $(\Phi,\Psi)$ is $O(N_{(\Phi,\Psi)} + k \log (N_{(\Phi,\Psi)}))$.
Using $\sum_{\text{$(\Phi,\Psi)$ shape of $v$}} N_{(\Phi,\Psi)} \le \pwidth^3$, we obtain that the total time for the merges corresponding to the shapes of $v$ is
$O(\pwidth^3 k)$.
Therefore, for each inner vertex $v$ of $T$, the table $\tab_v$ can be computed in time $O\pare{\pwidth^3 \card{\cons} + \pwidth^3 k \log(k)}$.
Since there are $\card{\vars}+\card{\cons}-1$ inner vertices of $T$, in total they require $O(\pwidth^3 \pare{\card{\vars}+\card{\cons}} \pare{\card{\cons} + k \log(k)})$ time.

\textbf{Root.}
From \cref{rem root vshape}, 
the content of $\tab_r$, indexed by $(\epsilon, \cons / \epsilon)$, implicitly contains a sorted list of $k$ highest-value assignments from $\vars$ to $\dom$ that satisfy the constraints in $\cons$.
Following our pointers, we can construct them explicitly in time $O\pare{k\pare{\card{\vars}+\card{\cons}}}$.
\end{prf}


\subsection{Weighted constraint violation}
\label{sec weighted constraint violation}

In this section we define a problem that significantly extends the weighted MaxSAT problem.
This problem is inherently defined on a separable system with $\consin = \emptyset$.
We remark that this problem can also be extended to general separable systems, but such an extension appears to offer little value and practical relevance.
In the \emph{weighted constraint violation problem,} we are given a separable system of the form $\pare{\vars,\dom,\emptyset,\cons}$, and $\wei^\con \in \R$ for every $\con \in \cons$.
For every assignment $\ass : \vars \to \dom$, we define its \emph{weight}
\begin{align*}
\wei(\ass) & := \sum_{\con \in \cons} \wei^\con \min\bra{\sum_{\var \in \vars} g^\con_\var(\ass(\var)), \gamma^\con} \\
& = \sum_{\con \in \cons} \wei^\con \min\bra{\cb^\con\pare{\ass}, \gamma^\con} \\
& = \sum_{\con \in \cons} \wei^\con (\cons/\ass)^\con.
\end{align*}
The goal is 
to find a highest-weight assignment from $\vars$ to $\dom$.

The name of the problem arises by considering the special case $\wei^\con \ge 1$, where for all $\con \in \cons$, we have $\wei(\ass) \le \sum_{\con \in \cons} \wei^\con \gamma^\con$, for every $\ass : \vars \to \dom$, and 
$\wei(\ass) = \sum_{\con \in \cons} \wei^\con \gamma^\con$ if and only if $\ass$ satisfies all the constraints in $\cons$.
Compared to the problems considered in \cref{sec counting,sec opt}, the weighted constraint violation problem might seem more exotic.
However, it contains as special cases MaxSAT and weighted MaxSAT, and it will allows us to show how our techniques can be used to significantly extend the known tractability of these problems for formulas with bounded PS-width \cite{SaeTelVat14}.
To write the weighted MaxSAT problem as a separable constraint problem, it suffices to observe that each clause can be written as a constraint $\con$ with $\dom := \bra{0,1}$, $\gamma^\con:=1$, and where $\wei^\con$ is the (non-negative) weight of $\con$ in the weighted MaxSAT problem.

While the approach that we use to solve the weighted constraint violation problem is still based on the theory of shapes that we developed in \cref{sec projections}, 
we will not be using the concept of \emph{assignments of shape,} which we introduced in \cref{sec shapes}, and that played a key role in \cref{sec counting,sec opt}.
Instead, we rely on the ``weaker'' notion of \emph{assignments of configuration,} which we define next.

Let $T$ be a branch decomposition of $\vars \cup \cons$, let $v$ be a vertex of $T$, and let $(\Phi,\Psi)$ be a shape for $v$.
We say that $\ass : \vars_v \to \dom$ has \emph{configuration} $\Phi$ if condition \eqref{cond s1 ge} holds, that is,
\begin{align*}
\Phi = \overline{\cons_v} / \ass.
\end{align*}
Furthermore, for $\ass : \vars_v \to \dom$ we define its \emph{$\Psi$-weight}
\begin{align*}
\wei^{\Psi} (\ass) := 
\sum_{\con \in \cons_v}
\wei^\con \min\bra{\cb^\con\pare{\ass}, \gamma^\con - \Psi^\con} \in \Z.
\end{align*}
In our algorithm, for the shape $(\Phi,\Psi)$, we will compute a highest-$\Psi$-weight assignment from $\vars_v$ to $\dom$ of configuration $\Phi$.
Such an assignment will be simple to compute in the leaves, and will yield the solution to the problem in the root, as discussed below.

\begin{remark}
\label{rem pB}
Consider a separable system of the form $\pare{\vars,\dom,\emptyset,\cons}$, let $T$ be a branch decomposition of $\vars \cup \cons$, and let $r$ be the root of $T$.
Consider the shape $(\epsilon, \cons / \epsilon)$ for $r$ (see \cref{rem proj} \ref{rem proj root phi} and \ref{rem proj root psi}),
and note that $\vars_r = \vars$.
Note that every assignment from $\vars$ to $\dom$ has configuration $\epsilon$.
Furthermore, $\cons / \epsilon$ is given explicitly by
\begin{align*}
&\pare{\cons / \epsilon}^\con = 0 && \forall \con \in \cons,
\end{align*}
thus
$\wei^{\cons/\epsilon}_r(\ass) = \wei(\ass)$
for every assignment $\ass$ from $\vars$ to $\dom$.
Therefore, the set of highest-$\pare{\cons / \epsilon}$-weight assignments from $\vars$ to $\dom$ of configuration $\epsilon$ coincides with the set of highest-weight assignments from $\vars$ to $\dom$.
\end{remark}

The next result shows how $\wei^{\Psi}\pare{\ass_1 \cup \ass_2}$ can be easily computed from $\wei^{\Psi_1} \pare{\ass_1}$ and $\wei^{\Psi_2} \pare{\ass_2}$;
It will be the key in computing the highest-$\Psi$-weight assignments, traversing $T$ in a bottom up manner.

\begin{lemma}
\label{lem sum wei}
Consider a separable system of the form $\pare{\vars,\dom,\emptyset,\cons}$, and let $\wei^\con \in \R$ for every $\con \in \cons$.
Let $T$ be a branch decomposition of $\vars \cup \cons$, and let $v$ be an inner vertex of $T$ with children $v_1$ and $v_2$.
Let $(\Phi,\Psi)$, $(\Phi_1,\Psi_1)$, $(\Phi_2,\Psi_2)$ be linked shapes for $v,v_1,v_2$.
Let $\ass_1,\assb_1 : \vars_{v_1} \to \dom$ of configuration $\Phi_1$ and $\ass_2,\assb_2 : \vars_{v_2} \to \dom$ of configuration $\Phi_2$.
Then, 
\begin{align*}
\wei^{\Psi}\pare{\ass_1 \cup \ass_2} 
= 
\wei^{\Psi_1} \pare{\ass_1} + 
\wei^{\Psi_2} \pare{\ass_2} + 
\sum_{\con \in \cons_{v_1}}
\wei^\con \min\bra{\Phi_2^\con,\gamma^\con-\Psi^\con} +
\sum_{\con \in \cons_{v_2}}
\wei^\con \min\bra{\Phi_1^\con,\gamma^\con-\Psi^\con}.
\end{align*}
\end{lemma}

\begin{prf}
It suffices to prove that, for every $\con \in \cons_{v_1}$,
\begin{align}
\label{eq claim min ind}
\min\bra{\cb^\con\pare{\ass_1 \cup \ass_2}, \gamma^\con - \Psi^\con} 
= \min\bra{\cb^\con\pare{\ass_1}, \gamma^\con - \Psi_1^\con} + \min\bra{\Phi_2^\con,\gamma^\con-\Psi^\con}.
\end{align}
In fact, from \eqref{eq claim min ind} and the symmetric identity for $v_2$, we obtain
\begin{align*}
\wei^{\Psi} \pare{\ass_1 \cup \ass_2} 
& = 
\sum_{\con \in \cons_v} \wei^\con \min\bra{\cb^\con\pare{\ass_1 \cup \ass_2}, \gamma^\con - \Psi^\con} \\
& =
\sum_{\con \in \cons_{v_1}}
\wei^\con \min\bra{\cb^\con\pare{\ass_1 \cup \ass_2}, \gamma^\con - \Psi^\con}
+
\sum_{\con \in \cons_{v_2}}
\wei^\con \min\bra{\cb^\con\pare{\ass_1 \cup \ass_2}, \gamma^\con - \Psi^\con} \\
& = 
\sum_{\con \in \cons_{v_1}}
\wei^\con \min\bra{\cb^\con\pare{\ass_1}, \gamma^\con - \Psi_1^\con} + 
\sum_{\con \in \cons_{v_1}}
\wei^\con \min\bra{\Phi_2^\con,\gamma^\con-\Psi^\con} \\
& \qquad + 
\sum_{\con \in \cons_{v_2}}
\wei^\con \min\bra{\cb^\con\pare{\ass_2}, \gamma^\con - \Psi_2^\con} + 
\sum_{\con \in \cons_{v_2}}
\wei^\con \min\bra{\Phi_1^\con,\gamma^\con-\Psi^\con} \\
& = 
\wei^{\Psi_1} \pare{\ass_1} + 
\wei^{\Psi_2} \pare{\ass_2} + 
\sum_{\con \in \cons_{v_1}}
\wei^\con \min\bra{\Phi_2^\con,\gamma^\con-\Psi^\con} +
\sum_{\con \in \cons_{v_2}}
\wei^\con \min\bra{\Phi_1^\con,\gamma^\con-\Psi^\con}.
\end{align*}

In the remainder of the proof, we show \eqref{eq claim min ind}.
First, we simplify the left hand side.
\begin{align*}
\min\bra{\cb^\con\pare{\ass_1 \cup \ass_2}, \gamma^\con - \Psi^\con}
& = \min\bra{\cb^\con\pare{\ass_1} + \cb^\con\pare{\ass_2}, \gamma^\con - \Psi^\con}
&& \text{(from \cref{obs vb} \ref{obs vb 3})} \\
& = \min\bra{\cb^\con\pare{\ass_1} + \min\bra{\cb^\con\pare{\ass_2},\gamma^\con}, \gamma^\con - \Psi^\con}
&& \text{(from \cref{lem 3min})} \\
& = \min\bra{\cb^\con\pare{\ass_1} + \pare{\overline{\cons_{v_2}} / \ass_2}^\con, \gamma^\con - \Psi^\con} \\
& = \min\bra{\cb^\con\pare{\ass_1} + \Phi_2^\con, \gamma^\con - \Psi^\con}
&& \text{($\ass_2$ of configuration $\Phi_2$)}.
\end{align*}
Next, we rewrite the first minimum on the right hand side using \eqref{cond l2}.
\begin{align*}
\min\bra{\cb^\con\pare{\ass_1}, \gamma^\con - \Psi_1^\con} 
& = \min\bra{\cb^\con\pare{\ass_1}, \gamma^\con - \min\bra{\Psi^\con + \Phi_2^\con, \gamma^\con}}.
\end{align*}
We can then rewrite \eqref{eq claim min ind} as follows:
\begin{align}
\label{eq min ind}
\min\bra{\cb^\con\pare{\ass_1}+\Phi_2^\con,\gamma^\con-\Psi^\con} 
- \min\bra{\cb^\con\pare{\ass_1},\gamma^\con-\min\bra{\Psi^\con+\Phi_2^\con,\gamma^\con}} 
= \min\bra{\Phi_2^\con,\gamma^\con-\Psi^\con}.
\end{align}
To check the identity \eqref{eq min ind}, we first consider the case $\Phi_2^\con \ge \gamma^\con - \Psi^\con$.
In this case, the right hand side of \eqref{eq min ind} equals $\gamma^\con - \Psi^\con$, and the left hand side equals
\begin{align*}
\pare{\gamma^\con-\Psi^\con} - \min\bra{\cb^\con\pare{\ass_1},\gamma^\con - \gamma^\con} 
= 
\gamma^\con-\Psi^\con.
\end{align*}
Next, we consider the case $\Phi_2^\con < \gamma^\con - \Psi^\con$.
In this case, the right hand side of \eqref{eq min ind} equals $\Phi_2^\con$ and the left hand side equals
\begin{align*}
\min\bra{\cb^\con\pare{\ass_1}+\Phi_2^\con,\gamma^\con-\Psi^\con} 
- \min\bra{\cb^\con\pare{\ass_1},\gamma^\con-\Psi^\con-\Phi_2^\con}.
\end{align*}
We add and subtract $\Phi_2^\con$, and bring $-\Phi_2^\con$ inside the first minimum.
\begin{align*}
\Phi_2^\con + \min\bra{\cb^\con\pare{\ass_1},\gamma^\con-\Psi^\con-\Phi_2^\con}
- \min\bra{\cb^\con\pare{\ass_1},\gamma^\con-\Psi^\con-\Phi_2^\con}
= \Phi_2^\con.
\end{align*}
\end{prf}

The following result is a direct consequence of \cref{lem sum wei}.

\begin{lemma}
\label{lem up best}
Consider a separable system of the form $\pare{\vars,\dom,\emptyset,\cons}$, and let $\wei^\con \in \R$ for every $\con \in \cons$.
Let $T$ be a branch decomposition of $\vars \cup \cons$, and let $v$ be an inner vertex of $T$ with children $v_1$ and $v_2$.
Let $(\Phi,\Psi)$, $(\Phi_1,\Psi_1)$, $(\Phi_2,\Psi_2)$ be linked shapes for $v,v_1,v_2$.
Let $\ass_1,\assb_1 : \vars_{v_1} \to \dom$ of configuration $\Phi_1$ and $\ass_2,\assb_2 : \vars_{v_2} \to \dom$ of configuration $\Phi_2$.
If $\wei^{\Psi_1} \pare{\assb_1} \le \wei^{\Psi_1} \pare{\ass_1}$ and $\wei^{\Psi_2} \pare{\assb_2} \le \wei^{\Psi_2} \pare{\ass_2}$, then $\wei^{\Psi} \pare{\assb_1 \cup \assb_2} \le \wei^{\Psi} \pare{\ass_1 \cup \ass_2}$.
\end{lemma}

\begin{prf}
From \cref{lem sum wei},
\begin{align*}
\wei^{\Psi}\pare{\assb_1 \cup \assb_2} 
& = 
\wei^{\Psi_1} \pare{\assb_1} + 
\wei^{\Psi_2} \pare{\assb_2} + 
\sum_{\con \in \cons_{v_1}}
\wei^\con \min\bra{\Phi_2^\con,\gamma^\con-\Psi^\con} +
\sum_{\con \in \cons_{v_2}}
\wei^\con \min\bra{\Phi_1^\con,\gamma^\con-\Psi^\con} \\
& \le
\wei^{\Psi_1} \pare{\ass_1} + 
\wei^{\Psi_2} \pare{\ass_2} + 
\sum_{\con \in \cons_{v_1}}
\wei^\con \min\bra{\Phi_2^\con,\gamma^\con-\Psi^\con} +
\sum_{\con \in \cons_{v_2}}
\wei^\con \min\bra{\Phi_1^\con,\gamma^\con-\Psi^\con} \\
& =
\wei^{\Psi}\pare{\ass_1 \cup \ass_2}.
\end{align*}
\end{prf}

We are now ready to present our algorithm for weighted constraint violation.

\begin{theorem}[Weighted constraint violation]
\label{th weighted constraint violation}
Consider a separable system of the form $\pare{\vars,\dom,\emptyset,\cons}$, let $\wei^\con \in \R$ for every $\con \in \cons$, and let $T$ be a branch decomposition of $\vars \cup \cons$ of projection-width $\pwidth$.
Then, the weighted constraint violation problem can be solved in time 
\begin{equation*}
O\pare{
\pwidth^3 \pare{\card{\vars}+\card{\cons}} \card{\cons}
+
\card{\vars} \card{\cons} \card{\dom}\log\pare{\card{\dom}}}.
\end{equation*}
\end{theorem}

\begin{prf}
First, we apply \cref{prop all proj} and compute $\proj(v)$ and $\proj(\overline{v})$, for every vertex $v$ of $T$, in time $O\pare{\pwidth^2 \log(\pwidth) \pare{\card{\vars}+\card{\cons}} \card{\cons} + \card{\vars}\card{\cons}\card{\dom}\log\pare{\card{\dom}}}$.

\textbf{Table.}
Next, our algorithm will construct, for each vertex $v$ of $T$, a table $\tab_v$ indexed by the shapes $(\Phi,\Psi)$ for $v$.
For a shape $(\Phi,\Psi)$, the content of the table $\tab_v$ at this index, which we denote by $\tab_v(\Phi,\Psi)$, should be a pair $(\ass,\wei^{\Psi}(\ass))$, where $\ass$ is a highest-$\Psi$-weight assignment from $\vars_v$ to $\dom$ of configuration $\Phi$.
We now explain how we can compute the table $\tab_v$, for every vertex $v$ of $T$.
This is done in a bottom up manner.

\textbf{Leaves corresponding to constraints.}
Consider a leaf $l$ of $T$ corresponding to a constraint $\con \in \cons$.
From \cref{rem proj} \ref{rem proj constraint-leaf}, there is only one assignment from $\vars_l$ to $\dom$, the empty assignment $\epsilon$, that can be constructed in time $O(1)$.
Clearly, there is only one $\Phi \in \proj(l)$, and it satisfies \eqref{cond s1 ge}.
Hence, we set $\tab_l(\Phi,\Psi) := (\epsilon,0)$ for every shape $(\Phi,\Psi)$ for $l$.
Therefore, we can compute the table $\tab_l$ in time $O\pare{\pwidth}$.
Since there are $\card{\cons}$ leaves corresponding to constraints in $\cons$, in total they require $O\pare{\pwidth \card{\cons}}$ time.

\textbf{Leaves corresponding to variables.}
Consider a leaf $l$ of $T$ corresponding to a variable $\var \in \vars$.
From \cref{rem proj} \ref{rem proj variable-leaf}, there are $\card{\dom}$ assignment from $\vars_l$ to $\dom$, and they can be constructed in time $O(\card{\dom})$.
We now fix one such assignment $\ass$, and observe that there is precisely one $\Phi \in \proj(l)$ such that $\ass$ has configuration $\Phi$, and it can be constructed in time $O(\cons)$.
Once we have considered all $\card{\dom}$ assignments, for every shape $(\Phi,\Psi)$ for $l$, which are at most $\pwidth^2$, we set $\tab_l(\Phi,\Psi) := (\ass,0)$, where $\ass$ is any assignment of configuration $\Phi$. 
Therefore, we can compute the table $\tab_l$ in time $O\pare{\card{\cons} \card{\dom} + \pwidth^2}$.
Since there are $\card{\vars}$ leaves corresponding to variables, in total they require $O\pare{\card{\vars} \card{\cons} \card{\dom} + \pwidth^2 \card{\vars}}$ time.

\textbf{Inner vertices.}
Consider now an inner vertex $v$ of $T$, with children $v_1,v_2$.
We loop over all triples $(\Phi,\Psi)$, $(\Phi_1,\Psi_1)$, $(\Phi_2,\Psi_2)$ of linked shapes for $v,v_1,v_2$, as described in the proof of \cref{th optimization}.
For each triple of linked shapes, let $\pare{\ass_1,\wei^{\Psi_1}(\ass_1)} := \tab_{v_1}(\Phi_1,\Psi_1)$ and $\pare{\ass_2,\wei^{\Psi_2}(\ass_2)} := \tab_{v_2}(\Phi_2,\Psi_2)$.
We then construct the $\Psi$-weight of the ``candidate assignment for $(\Phi,\Psi)$'' given by $\ass_1 \cup \ass_2$, and this can be done in time $O(\card{\cons})$ due to \cref{lem sum wei}.

For each shape $(\Phi,\Psi)$ for $v$, we then set the highest-$\Psi$-weight assignment $\ass_1 \cup \ass_2$, among all candidate assignments for $(\Phi,\Psi)$, as the content of $\tab_v(\Phi,\Psi)$, together with its $\Psi$-weight.
To improve runtime, here we do not explicitly construct $\ass_1 \cup \ass_2$; instead, we store pointers to $(\Phi_1,\Psi_1)$ and $(\Phi_2,\Psi_2)$ giving the highest-$\Psi$-weight, so that this is done in time $O(1)$ instead of $O\pare{\card{\vars}}$.
Therefore, for each inner vertex $v$ of $T$, the table $\tab_v$ can be computed (partially implicitly) in time $O(\pwidth^3 \card{\cons})$.
Since there are $\card{\vars}+\card{\cons}-1$ inner vertices of $T$, in total they require $O(\pwidth^3 \pare{\card{\vars}+\card{\cons}} \card{\cons})$ time.

We now show that we set $\tab_v$ correctly, for every inner vertex $v$ of $T$.
Since we already proved it for the leaves of $T$, we now assume inductively that $\tab_{v_1}$ and $\tab_{v_2}$ have been set correctly, where $v_1,v_2$ are the children of $v$.
Let $(\Phi,\Psi)$ be a shape for $v$, and let $\assb$ be an assignment from $\vars_v$ to $\dom$ of configuration $\Phi$. 
Let $\assb_1$ and $\assb_2$ denote the restrictions of $\assb$ to $\vars_{v_1}$ and $\vars_{v_2}$, respectively, and set $\Phi_1 := \overline{\cons_{v_1}} / \assb_1$ and $\Phi_2 := \overline{\cons_{v_2}} / \assb_2$.
Consider now our procedure, when it considers the triple of linked shapes $(\Phi,\Psi)$, $(\Phi_1,\Psi_1)$, $(\Phi_2,\Psi_2)$ originated from $\Psi$, $\Phi_1$, $\Phi_2$.
Note that, according to \cref{lem phi ass} and \eqref{cond l1}, the linked shape $(\Phi,\Psi)$ that we just obtained is indeed the shape for $v$ we started from.
Let $\ass_1$ and $\ass_2$ be the assignments in $\tab_{v_1}(\Phi_1,\Psi_1)$ and $\tab_{v_2}(\Phi_2,\Psi_2)$, respectively.
So $\ass_1 \cup \ass_2$ is a candidate assignment for $\pare{\Phi,\Psi}$.
By induction, $\wei^{\Psi_1}\pare{\assb_1} \le \wei^{\Psi_1}\pare{\ass_1}$ and $\wei^{\Psi_2}\pare{\assb_2} \le \wei^{\Psi_2}\pare{\ass_2}$.
\cref{lem up best} then implies $\wei^{\Psi}(\assb) \le \wei^{\Psi}(\ass)$.
Now let $\assc : \vars_v \to \dom$ be the assignment in $\tab_v(\Phi,\Psi)$.
$\assc$ is a candidate assignment for $(\Phi,\Psi)$, so it has configuration $\Phi$, due to \cref{lem phi ass} and \eqref{cond l1}.
Furthermore, by construction, we have $\wei^{\Psi}(\ass) \le \wei^{\Psi}(\assc)$, therefore $\wei^{\Psi}(\assb) \le \wei^{\Psi}(\assc)$.

\textbf{Root.}
Once the table $\tab_v$ is computed, for every vertex $v$ of $T$, we consider the root $r$ of $T$, and the shape $\pare{\epsilon, \cons/\epsilon}$ for $r$ (see \cref{rem proj} \ref{rem proj root phi} and \ref{rem proj root psi}).
From \cref{rem pB}, the table $\tab_r$, indexed by $\pare{\epsilon, \cons/\epsilon}$, implicitly contains a highest-weight assignment from $\vars$ to $\dom$.
Following our pointers, we can construct it explicitly in time $O\pare{\card{\vars}+\card{\cons}}$.
\end{prf}



\section{Some consequences}
\label{sec consequences}

In this section, we obtain some corollaries of our main theorems, and we discuss
how our results subsume previously known tractability results in integer linear optimization,
binary polynomial optimization, and Boolean satisfiability.



\subsection{Main consequences}
\label{sec consequences projection-width}

In \cref{sec INLP}, we specialize some of our main results to integer separable (nonlinear) optimization; in \cref{sec ILP}, to integer linear optimization; in \cref{sec BPO}, to binary polynomial optimization; and in \cref{sec SAT}, to Boolean satisfiability.
We focus in particular on the optimization problems, emphasizing the consequences of \cref{th optimization}.
All the results that we obtain are new, except for Boolean satisfiability, where we recover precisely the tractability of weighted MaxSAT and \#SAT for CNF formulas with bounded PS-width in~\cite{SaeTelVat14}.




\subsubsection{Integer separable optimization}
\label{sec INLP}

The integer separable optimization problem is the special case of the optimization problem considered in this work (\cref{sec opt}), where the domain consists of the integer points in a bounded interval, and only inequality constraints are allowed.
Formally, a \emph{separable inequality system} is a quadruple $\pare{\vars,\domI,\consSle,\consSge}$, where $\vars$ is a finite set of \emph{variables}, $\domI$ is a finite \emph{domain} set of the form $\domI = \bra{-\dom_{\max},-\dom_{\max}+1,\dots,\dom_{\max}}$ for some $\dom_{\max} \in \Z_{\ge 0}$, and where $\consSle$, $\consSge$ are sets of \emph{separable inequality constraints} of the form
\begin{align*}
& \sum_{\var \in \vars} f^\con_\var(\var) \le \delta^\con && \con \in \consSle, \\
& \sum_{\var \in \vars} f^\con_\var(\var) \ge \delta^\con && \con \in \consSge,
\end{align*}
where $f^\con_\var : \domI \to \Z$ for every $\var \in \vars$ and $\con \in \consSle \cup \consSge$, 
and where $\delta^\con \in \Z$ for every $\con \in \consSle \cup \consSge$.
Clearly, a constraint in $\consSge$ can also be expressed as a constraint in $\consSle$, and vice versa.
However, we choose to keep both types of constraints, as moving one inequality from one set to the other may affect the resulting projection-width of the system.

The definition of projection-width of a separable inequality system follows from our original definition for separable systems in \cref{sec def projection-width}, since every inequality constraint $\con \in \consSle$ can be written as a set constraint in $\consin$ with $\Gamma^\con = \bra{0,1,\dots,\gamma^\con}$, where as usual
\begin{align*}
\gamma^\con & := \delta^\con - \sum_{\var \in \vars} \min\bra{f^\con_\var(d) \mid d \in \domI}.
\end{align*}

In the \emph{integer separable optimization problem,} we are given a separable inequality system $\pare{\vars,\domI,\consSle,\consSge}$ and $\nu_\var : \domI \to \R$ for every $\var \in \vars$.
For every assignment $\ass : \vars \to \dom$, we define its \emph{value}
\begin{align*}
\val(\ass) & := \sum_{\var \in \vars} \nu_\var \pare{\ass(\var)}.
\end{align*}
The goal is to find a highest-value assignment from $\vars$ to $\domI$ satisfying the constraints in $\consSle \cup \consSge$, or prove that no such assignment exists.

Since $\card{\domI} = 2 \dom_{\max} + 1$, our \cref{th optimization} directly implies the following result:

\begin{corollary}[Integer separable optimization]
\label{cor inl optimization}
Consider a separable inequality system $\pare{\vars,\domI,\consSle,\consSge}$, let $\consS := \consSle \cup \consSge$, and let $T$ be a branch decomposition of $\vars \cup \consS$ of projection-width $\pwidth$.
Let $\nu_\var : \dom \to \R$ for every $\var \in \vars$.
Then, the integer separable optimization problem can be solved in time
\begin{equation}
\label{eq inl opt runtime}
O\pare{
\pwidth^3 \pare{\card{\vars}+\card{\consS}} \card{\consS}
+
\card{\vars} \card{\consS} \dom_{\max} \log\pare{\dom_{\max}}}.
\end{equation}
\end{corollary}

\subsubsection{Integer linear optimization}
\label{sec ILP}

The integer linear optimization problem is the special case of integer separable optimization considered in \cref{sec INLP}, where all functions are linear.

Formally, a \emph{linear inequality system} is a separable
inequality system $\pare{\vars,\domI,\consLle,\consLge}$, where for every $\con \in \consLle \cup \consLge$ and $\var \in \vars$, the function $f^\con_\var : \domI \to \Z$ is of the form 
\begin{align*}
& f^\con_\var(\var) = a^\con_\var \var,
\end{align*}
for some $a^\con_\var \in \Z$.
In the \emph{integer linear optimization problem,} we are given a linear inequality system $\pare{\vars,\domI,\consLle,\consLge}$ and $\nu_\var \in \R$ for every $\var \in \vars$.
For every assignment $\ass : \vars \to \domI$, we define its \emph{value}
\begin{align*}
\val(\ass) & := \sum_{\var \in \vars} \nu_\var \cdot \pare{\ass(\var)}.
\end{align*}
The goal is to find a highest-value assignment from $\vars$ to $\domI$ satisfying the constraints in $\consLle \cup \consLge$, or prove that no such assignment exists.
\cref{cor inl optimization} directly implies the following result:

\begin{corollary}[Integer linear optimization]
\label{cor il optimization}
Consider a linear inequality system $\pare{\vars,\domI,\consLle,\consLge}$, let $\consL := \consLle \cup \consLge$, and let $T$ be a branch decomposition of $\vars \cup \consL$ of projection-width $\pwidth$.
Let $\nu_\var \in \R$ for every $\var \in \vars$.
Then, the integer linear optimization problem can be solved in time
\begin{equation}
\label{eq il opt runtime}
O\pare{
\pwidth^3 \pare{\card{\vars}+\card{\consL}} \card{\consL}
+
\card{\vars} \card{\consL} \dom_{\max} \log\pare{\dom_{\max}}}.
\end{equation}
\end{corollary}

\subsubsection{Binary polynomial optimization}
\label{sec BPO}

Important applications of our results arise in \emph{binary polynomial optimization}, an area that has recently seen significant progress; see, for example,  \cite{dPKha17MOR,dPKha18MPA,dPKha18SIOPT,dPKha21MOR,dPKha24MPA}.
In a \emph{binary polynomial optimization problem}, we are given a hypergraph $H=(V,E)$, together with $\val_v \in \R$ for every $v \in V$, and $\val_e \in \R$ for every $e \in E$.
Let $\vars_V := \bra{\var_v \mid v \in V}$ denote the set of variables.
The goal is to find an assignment $\ass : \vars_V \to \bra{0,1}$ maximizing
\begin{align*}
\sum_{v \in V} \val_v \ass(\var_v) + \sum_{e \in E} \val_e \prod_{v \in e} \ass(\var_v).
\end{align*}
The objective function above is, in general, not separable.
However, it is well-known how the binary polynomial optimization problem can be reformulated as an integer linear optimization problem.
To do so, we apply Fortet’s linearization \cite{For59,For60}, introducing auxiliary variables $\yvars_E = \bra{\yvar_e \mid e \in E}$.
For every assignment $\ass : \vars_V \cup \yvars_E \to \bra{0,1}$, we define its \emph{value} by the linear function
\begin{align*}
\val(\ass) := \sum_{v \in V} \val_v \ass(\var_v) + \sum_{e \in E} \val_e \ass(\yvar_e).
\end{align*}
Consistency between the auxiliary variables and the original ones is then enforced through linear inequalities.
Among several possible formulations, we adopt the \emph{standard linearization,} though exploring alternative formulations could be an interesting direction in light of the results of this paper.

As we discussed in \cref{sec INLP}, each linear inequality can be placed in either $\consLle$ or $\consLge$, and this choice may affect the resulting projection-width.
Here, we place all inequalities in $\consLge$.
This yields the linear inequality system $S_{\BPO} = \pare{\vars_V \cup \yvars_E,\bra{0,1},\emptyset,\consBPO}$, where the constraints in $\consBPO$ are given by
\begin{equation*}
\begin{aligned}
& (1-\yvar_e) + \var_v \ge 1 && \forall v \in e, \ \forall e \in E, \\
& \sum_{v \in e} (1-\var_v) + \yvar_e \ge 1 && \forall e \in E. 
\end{aligned}
\end{equation*}
The binary polynomial optimization problem is thus equivalent to the integer linear optimization problem of finding a maximum-value assignment from $\vars_V \cup \yvars_E$ to $\bra{0,1}$ satisfying the constraints in $\consBPO$.
We can then directly apply \cref{cor il optimization} to obtain the following tractability results for binary polynomial optimization, where we denote by $\size(H) := \sum_{e \in E} \card{e}$ the size of the hypergraph $H$.

\begin{corollary}[Binary polynomial optimization]
\label{cor BPO}
Consider a hypergraph $H=(V,E)$ and let $T$ be a branch decomposition of $\vars_V \cup \yvars_E \cup \consBPO$ of projection-width $\pwidth$.
Let $\val_v \in \R$ for every $v \in V$, and $\val_e \in \R$ for every $e \in E$.
Then, the binary polynomial optimization problem can be solved in time 
\begin{align*}
O\pare{\pwidth^3 \pare{\card{V} + \size(H)} \size(H)}.
\end{align*}
\end{corollary}



\cref{th optimization} also allows us to solve classes of \emph{constrained binary polynomial optimization,} which is the extension of binary polynomial optimization obtained by considering only assignments satisfying some given additional separable constraints.
\footnote{If the additional separable constraints are linear, it suffices to use \cref{cor il optimization} instead of \cref{th optimization}.}

Formally, let $\pare{\vars_V \cup \yvars_E,\bra{0,1},\consinc,\consgec}$ be a separable system representing these constraints.
We then define the \emph{constrained binary polynomial optimization problem} as the problem of finding a maximum-value assignment from $\vars_V \cup \yvars_E$ to $\bra{0,1}$ satisfying the constraints in $\consinc \cup \consgec \cup \consBPO$.
At this point, we can again directly apply \cref{th optimization} and obtain the following result.

\begin{corollary}[Constrained binary polynomial optimization]
\label{cor constrained BPO}
Consider a hypergraph $H=(V,E)$, a separable system $\pare{\vars_V \cup \yvars_E,\bra{0,1},\consinc,\consgec}$, and let $T$ be a branch decomposition of $\vars_V \cup \yvars_E \cup \consinc \cup \consgec \cup \consBPO$ of projection-width $\pwidth$.
Then, the constrained binary polynomial optimization problem can be solved in time
\begin{align*}
O\pare{\pwidth^3 \pare{\card{V} + \size(H)} \size(H)}.
\end{align*}
\end{corollary}

We remark that literals can be naturally incorporated into our framework, in both the unconstrained and constrained settings.
The only modification required is that, in the constraints in $\consBPO$, some variables $\var_v$ may be replaced by $1-\var_v$ and vice versa. 
Consequently, all results in this section remain valid in this more general setting.
We have excluded literals purely for notational simplicity.
For additional background on binary polynomial optimization with literals, also known as \emph{pseudo-Boolean optimization,} we refer the reader to \cite{BorHam02,CapdPDiG24xx,dPKha24xx,dPKha25MPA}.

Note that most known tractability results for binary polynomial optimization rely on structural properties of the underlying hypergraph $H$.
In contrast, \cref{cor BPO,cor constrained BPO} are of a different nature, as they are based instead on properties of the constraint system.
In \cref{sec BPO itw}, we present an example illustrating how these these results can be leveraged to derive tractability conditions expressed in terms of the hypergraph, thereby bringing them closer in spirit to existing results in the literature.

\subsubsection{Boolean satisfiability}
\label{sec SAT}

In this section, we show that our main results subsume those of \cite{SaeTelVat14} on the tractability of weighted MaxSAT and \#SAT for CNF formulas with bounded PS-width.
This influential result accounts for nearly all known tractable cases of these SAT problems, including formulas with 
bounded primal treewidth, 
bounded incidence treewidth, 
bounded signed incidence clique-width, 
bounded incidence clique-width, 
bounded MIM-width, 
as well as $\gamma$-acyclic formulas and disjoint-branches formulas~\cite{SaeTelVat14,Cap16PhD}.

We begin by defining PS-width, where “PS” stands for \emph{precisely satisfiable}.
Our notation closely parallels that used for projection-width, highlighting the structural similarity between the two concepts.
A \emph{CNF formula} is a pair $(\vars,F)$, where $\vars$ is a set of variables and $F$ is a set of clauses.
Recall that a \emph{clause} is the disjunction of literals, that is, variables or negation of variables of $\vars$.
Given $\vars' \subseteq \vars$, $F' \subseteq F$, and $\ass : \vars' \to \bra{0,1}$, we define 
$F' / \ass$, as the set of clauses in $F'$ that are satisfied by $\ass$.
Given $\vars' \subseteq \vars$ and $F' \subseteq F$, we denote by 
\begin{align*}
\proj(F', \vars')=\{F' / \ass \mid \ass: \vars' \rightarrow\bra{0,1}\}.
\end{align*}
Given a branch decomposition $T$ of $\vars \cup F$ and a vertex $v$ of $T$, we denote by $T_v$ the subtree of $T$ rooted in $v$, by $F_v$ the set clauses of $F$ such that the corresponding vertex appears in the leaves of $T_v$ and by $\vars_v$ the set of variables of $F$ that similarly appear in the leaves of $T_v$. 
We denote by 
\begin{align*}
\proj(v) &:= \proj\pare{F \setminus F_v, \vars_v}, \\
\proj(\overline v) &:= \proj\pare{F_v, \vars \setminus \vars_v}.
\end{align*}
If we denote by $V(T)$ the set of vertices of $T$, the \emph{PS-width} of the formula $\pare{\vars,F}$ and $T$ is defined by
\begin{align*}
\max _{v \in V(T)} \max \pare{\card{\proj(v)},\card{\proj(\overline v)}}.
\end{align*}
The \emph{PS-width} of a formula $\pare{\vars,F}$ is then defined as the minimum among the PS-widths of the formula $\pare{\vars,F}$ and $T$, over all branch decompositions $T$ of $\vars \cup F$.

\begin{observation}
\label{obs PS to PSV}
Let $\pare{\vars,F}$ be a CNF formula.
Then, there exists a separable system $\pare{\vars,\dom,\emptyset,\consF}$ with $\card{\consF} = \card{F}$, and $\dom=\bra{0,1}$ such that an assignments $\ass : \vars \to \bra{0,1}$ satisfies the constraints in $\consF$ if and only if it satisfies the clauses in $F$.
Moreover, a branch decomposition of $\vars \cup F$ of PS-width $\pswidth$ yields a branch decomposition of $\vars \cup \consF$ of projection-width $\pswidth$.
Furthermore, we have $g^\con_\var(\var) \in \bra{0, \var, 1-\var}$ for every $\var \in \vars$, $\con \in \consF$, and $\gamma^\con=1$ for every $\con \in \consF$.
\end{observation}

\begin{prf}
For every clause in $F$, we write a constraint in $\consF$ over variables $\vars$ with domain $\dom = \bra{0,1}$ defined by
\begin{align*}
\sum_{\var \in \vars^+} \var + \sum_{\var \in \vars^-} (1-\var) \ge 1,
\end{align*}
where $\vars^+$ denotes the set of variables that appear in the positive literals in the clause, and $\vars^-$ denotes the set of variables that appear in the negative literals in the clause.
Clearly, an assignment $\ass : \vars \to \bra{0,1}$ satisfies the clause if and only if it satisfies the obtained constraint.
It then suffices to show that a branch decomposition of $\vars \cup F$ of PS-width $\pswidth$ yields a branch decomposition of $\vars \cup \consF$ of projection-width $\pswidth$.
To see this, let $\vars' \subseteq \vars$, $F' \subseteq F$, $\ass: \vars' \rightarrow\bra{0,1}$, and let $\cons'$ be the set of constraints in $\consF$ originating from the clauses in $F'$.
Observe that $\pare{\cons' / \ass}^\con \in \bra{0,1}$ for every $\con \in \consF$.
Furthermore, a clause $f \in F'$ is in the set $F' / \ass$ if and only if $\pare{\cons' / \ass}^\con = 1$, where $\con$ is the constraint in $\cons'$ corresponding to $f$.
\end{prf}

Thanks to \cref{obs PS to PSV}, in the special case where $\consin = \emptyset$, $\dom = \bra{0,1}$, $g^\con_\var(\var) \in \bra{0, \var, 1-\var}$ for every $\var \in \vars$, $\con \in \consF$, and $\gamma^\con=1$ for every $\con \in \consF$, \cref{th counting} recovers the tractability of \#SAT on CNF formulas with bounded PS-width as established in~\cite{SaeTelVat14}.
Similarly, \cref{th weighted constraint violation} subsumes the tractability of weighted MaxSAT for CNF formulas with bounded PS-width in the same work.
The running time we obtain for these two problems is comparable to that in~\cite{SaeTelVat14}, and is given by
\begin{equation*}
O\pare{\pswidth^3 \pare{\card{\vars}+\card{F}} \card{F}},
\end{equation*}
where $\pswidth$ denotes the PS-width of the given branch decomposition.

\subsection{Consequences for incidence treewidth}
\label{sec consequences incidence treewidth}

In this section, we show that our main results imply the tractability of optimization and counting problems over separable systems whose incidence graph has bounded treewidth.
Consider a separable system $S = \pare{\vars,\dom,\consin,\consge}$, and let $\cons := \consin \cup \consge$.
The \emph{incidence graph} of $S$, denoted $G_{\inc}\pare{S}$, is the bipartite graph with vertex bipartition $\vars \cup \cons$, where an edge connects $\var \in \vars$ and $\con \in \cons$ if and only if $g^\con_\var \not\equiv 0$.

In the next result, we show that, given a tree decomposition of $G_{\inc}\pare{S}$ with treewidth bounded by a constant, one can efficiently construct a branch decomposition of $\vars \cup \cons$ whose projection-width is polynomially bounded.

\begin{lemma}[Treewidth and projection-width]
\label{lem itw PSw}
Consider a separable system $S = \pare{\vars,\dom,\consin,\consge}$, let $\cons := \consin \cup \consge$, let $\gamma := \max\bra{\gamma^\con \mid \con \in \cons}$, and let $T$ be a tree decomposition of $G_{\inc}\pare{S}$ of treewidth $\itwidth$.
Then, in time $O(\card{\vars} + \card{\cons})$, we can construct a branch decomposition $T'$ of $\vars \cup \cons$ of projection-width at most 
\begin{align*}
\max\bra{\card{\dom},\gamma+1}^{\itwidth+1}.
\end{align*}
\end{lemma}

\begin{prf}
Without loss of generality, we can assume that $T$ is binary.
In fact, for any vertex with more than two children, we can replace it with a binary tree of new vertices, each associated with the same bag, i.e., subset of $\vars \cup \cons$.
This transformation does not increase the size of any bag and therefore does not increase the treewidth.
We can then construct a branch decomposition $T'$ of $\vars \cup \cons$ as follows.
First, we add a vertex $r$ as the father of the root of $T$.
Then, for every $v \in \vars \cup \cons$, let $t$ be the vertex of $T$ closest to the root of $T$ such that $v$ appears in the bag of $t$, and we hang a leaf with label $v$ on the edge between $t$ and its father.
We then remove the leaves that have no label.
The resulting tree $T'$ is binary and for every $v \in \vars \cup \cons$, there exists precisely one leaf of $T'$ with label $v$.
Therefore, $T'$ is a branch decomposition of $\vars \cup \cons$.
We claim that the projection-width of $S$ and $T'$ is at most the one in the statement.

First, let $l$ be a leaf of $T'$ corresponding to a constraint $\con$ in $\cons$.
Since $\vars_l = \emptyset$, we have $\card{\proj(l)} = 1$.
Moreover, $\cons_l = \bra{\con}$ implies $\card{\proj(\overline{l})} \le \gamma^c+1$.
Next, let $l$ be a leaf of $T'$ corresponding to a variable $\var$ in $\vars$.
Since $\vars_l = \bra{\var}$, we have $\card{\proj(l)} \le \card{\dom}$.
Moreover, $\cons_l = \emptyset$ implies $\card{\proj(\overline{l})}=1$.

Next, let $v$ be an inner vertex of $T'$.
If $v$ is also a vertex of $T$, let $t := v$.
Otherwise, let $t$ be the vertex of $T$ in $T'_v$ closest to $v$.
We observe that, by construction, the label of every leaf of $T'_v$ appears only in bags of $T_t$.
Moreover, the label of every leaf of $T' \setminus T'_v$ either only appears in bags of $T \setminus T_t$, or it appears in the bag of $t$.

Let $\var \in \vars_v$ and $\con \in \overline{\cons_v}$ such that $g^\con_\var \not\equiv 0$.
By the previous observation, $\var$ appears only in bags of $T_t$ and $\con$ only appears in bags of $T \setminus T_t$ or in the bag of $t$.
Since $g^\con_\var \not\equiv 0$, $\var$ and $\con$ must appear in a common bag, so $\con$ must appear in the bag of $t$.
Therefore, the constraints $\con \in \overline{\cons_v}$ with $g^\con_\var \not\equiv 0$ for some $\var \in \vars_v$ all appear in the bag corresponding to $t$, so they are at most $\itwidth+1$.
Hence, $\card{\proj(v)} \le \pare{\gamma+1}^{\itwidth+1}$.

Let $\con \in \cons_v$ and $\var \in \overline{\vars_v}$ such that $g^\con_\var \not\equiv 0$.
By the previous observation, $\con$ appears only in bags of $T_t$ and $\var$ only appears in bags of $T \setminus T_t$ or in the bag of $t$.
Since $g^\con_\var \not\equiv 0$, $\var$ and $\con$ must appear in a common bag, so $\var$ must appear in the bag of $t$.
Therefore, the variables $\var \in \overline{\vars_v}$ with $g^\con_\var \not\equiv 0$ for some $\con \in \cons_v$ all appear in the bag corresponding to $t$, so they are at most $\itwidth+1$.
Hence, $\card{\proj(v)} \le \card{\dom}^{\itwidth+1}$.

Thus the projection-width of $T'$ is at most $\max\bra{\card{\dom},\gamma+1}^{\itwidth+1}$.
\end{prf}

\cref{lem itw PSw}, together with our main results \cref{th counting,th top-k,th optimization,th weighted constraint violation}, implies the tractability of problems defined over separable systems $S$ whose incidence graph $G_{\inc}\pare{S}$ has bounded treewidth.
In particular, from \cref{th optimization,lem itw PSw} we obtain the following result.

\begin{corollary}[Optimization -- incidence treewidth]
\label{cor optimization itw}
Consider a separable system $S = \pare{\vars,\dom,\consin,\consge}$, let $\cons := \consin \cup \consge$, let $\gamma := \max\bra{\gamma^\con \mid \con \in \cons}$, and let $T$ be a tree decomposition of $G_{\inc}\pare{S}$ of treewidth $\itwidth$.
Let $\nu_\var : \dom \to \R$ for every $\var \in \vars$.
Then, the optimization problem can be solved in time \eqref{eq opt runtime}, where $\pwidth$ is replaced by $\max\bra{\card{\dom},\gamma+1}^{\itwidth+1}.$
\end{corollary}

It is natural to ask whether one could design an algorithm with a running time similar to that in \cref{cor optimization itw}, but where $\gamma$ is replaced by $\log(\gamma)$.
However, this is unlikely, as it would yield a polynomial-time algorithm for the subset sum problem, or for the $0,1$ knapsack problem, both of which are well-known to be weakly NP-complete. 
Indeed, such problems can be naturally expressed as optimization problems with $\dom = \bra{0,1}$ and $\itwidth=1$, since they involve only a single constraint.

It is worth emphasizing that the bounded incidence treewidth setting in \cref{cor optimization itw} is substantially more restrictive than our main bounded projection-width setting in \cref{th optimization}.
Even within the more limited SAT framework discussed in \cref{sec SAT}, several important classes of formulas admit PS-width polynomially bounded, but incidence treewidth not bounded by a constant, including those with 
bounded signed incidence clique width, 
bounded incidence clique-width, 
bounded MIM-width, 
$\gamma$-acyclic formulas, and 
formulas with disjoint branches
\cite{SaeTelVat14,Cap16PhD}.
In \cref{obs family} below, we present a family of linear inequality systems whose incidence treewidth is not bounded by a constant, while the projection-width remains polynomially bounded.
For this family, \cref{th optimization} yields a polynomial-time algorithm, while \cref{cor optimization itw} does not.
To the best of our knowledge, also \cref{cor optimization itw} is new.
The main reason we wrote \cref{cor optimization itw} is that it allows us to illustrate how our results subsume known tractability results in integer linear optimization (see \cref{sec ILP incidence}) and binary polynomial optimization (see \cref{sec BPO itw}).

\begin{observation}
\label{obs family}
There exists a family of linear inequality systems $S_L$ with $2n$ variables such that the treewidth of $G_{\inc}\pare{S_L}$ is at least $n$, and the projection-width of $S_L$ is at most $2n$.
\end{observation}

\begin{prf}
Consider the linear inequality system $S_L = \pare{\vars,\domI,\emptyset,\consLge}$, with variables $\vars=\bra{\var_1,\var_2,\dots,\var_{2n}}$, domain $\domI = \bra{0,1}$, and linear inequality constraints $\consLge = \bra{\con_1,\con_2,\dots,\con_{2n}}$, where for $k \in \bra{1,2,\dots,2n}$, $\con_{k}$ is given by
\begin{align*}
\sum_{i=1}^k \var_i \ge \gamma^{\con_k},
\end{align*}
where $\gamma^{\con_k} \in \bra{0,1,\dots,2n}$.
The subgraph of $G_{\inc}\pare{S_L}$ induced by variables $\var_{1},\var_{2}, \dots,\var_{n}$ and constraints $\con_{n+1},\con_{n+2}, \dots,\con_{2n}$ is complete bipartite, thus the treewidth of $G_{\inc}\pare{S_L}$ is at least $n$.
Now let $T$ be the linear branch decomposition of $\vars \cup \consLge$ in \cref{fig branch}.
It is simple to show that the projection-width of $S_L$ and $T$ is bounded by $2n$.
\end{prf}

\begin{figure}
\center
\caption{Illustration of the branch decomposition $T$ in the proof of \cref{obs family}.}
\label{fig branch}
\scalebox{0.8}{%
\begin{forest}
for tree={
  parent anchor=center,
  child anchor=center,
  align=center,
  edge={-},
  l sep=0pt,
  s sep=60pt,
}
[$\bullet$
[$\bullet$, label=below:$\con_{2n}$]
[$\bullet$
[$\bullet$, label=below:$\var_{2n}$]
[$\bullet$
[$\bullet$, label=below:$\con_{2n-1}$]
[$\bullet$
[$\bullet$, label=below:$\var_{2n-1}$]
[$\bullet$
[$\bullet$, label=below:$\cdots$]
[$\bullet$
[$\bullet$, label=below:$\con_2$]
[$\bullet$
[$\bullet$, label=below:$\var_2$]
[$\bullet$, label=below:$\var_1$]
]
]
]
]
]
]
]
\end{forest}
}
\end{figure}

We note that \cref{cor optimization itw} immediately yields several classical tractability results.
Given a hypergraph $H=(V,E)$, recall that the \emph{incidence graph} of $H$, denoted $G_{\inc}\pare{H}$, is the bipartite graph with vertex bipartition $V \cup E$, where an edge connects $v \in V$ and $e \in E$ if and only if $v \in e$.
\cref{cor optimization itw} implies that, given a tree decomposition of the incidence graph $G_{\inc}(H)$ of treewidth $\itwidth$, the weighted set cover and weighted set packing problems on $H$ can be solved in time $O(2^{3(\itwidth+1)}\pare{\card{E}+\card{V}}\card{V})$, while the weighted hitting set and weighted independent set problems on $H$ can be solved in time $O(2^{3(\itwidth+1)}\pare{\card{V}+\card{E}}\card{E})$.


Finally, we note that \cref{cor optimization itw} remains valid when the incidence treewidth is replaced by the primal treewidth.
More precisely, the same result holds if, in its statement, the incidence graph of the separable system $G_{\inc}\pare{S}$ is replaced by the primal graph $G_{\pri}\pare{S}$.
The \emph{primal graph} of a separable system $S$ is the graph $G_{\pri}\pare{S}$ with vertex set $\vars$, where two vertices $\var,\var’ \in \vars$ are adjacent if and only if there exists a constraint $\con \in \cons$ such that $g^\con_\var \not\equiv 0$ and $g^\con_{\var’} \not\equiv 0$.
This follows from the well-known fact that any tree decomposition of $G_{\pri}\pare{S}$ of treewidth $\ptwidth$ can be transformed into a tree decomposition of $G_{\inc}\pare{S}$ of treewidth at most $\ptwidth + 1$.

\subsubsection{Integer separable optimization}
\label{sec INLP incidence}

From \cref{cor inl optimization,lem itw PSw} we obtain the following result.

\begin{corollary}[Integer separable optimization -- incidence treewidth]
\label{cor inl optimization itw}
Consider a separable inequality system $S_S = \pare{\vars,\domI,\consSle,\consSge}$, let $\consS := \consSle \cup \consSge$, let $\gamma := \max\bra{\gamma^\con \mid \con \in \consS}$, and let $T$ be a tree decomposition of $G_{\inc}\pare{S_S}$ of treewidth $\itwidth$.
Let $\nu_\var : \dom \to \R$ for every $\var \in \vars$.
Then, the integer separable optimization problem can be solved in time \eqref{eq inl opt runtime}, where $\pwidth$ is replaced by $\max\bra{\card{\domI},\gamma+1}^{\itwidth+1}.$
\end{corollary}

\subsubsection{Integer linear optimization}
\label{sec ILP incidence}

From \cref{cor il optimization,lem itw PSw} we obtain the result below.

\begin{corollary}[Integer linear optimization -- incidence treewidth]
\label{cor il optimization incidence}
Consider a linear inequality system $S_L = \pare{\vars,\domI,\consLle,\consLge}$, let $\consL := \consLle \cup \consLge$, let $\gamma := \max\bra{\gamma^\con \mid \con \in \consL}$, and let $T$ be a tree decomposition of $G_{\inc}\pare{S_L}$ of treewidth $\itwidth$.
Let $\nu_\var \in \R$ for every $\var \in \vars$.
Then, the integer linear optimization problem can be solved in time
\eqref{eq il opt runtime}, where $\pwidth$ is replaced by $\max\bra{2 \dom_{\max}+1,\gamma+1}^{\itwidth+1}.$
\end{corollary}

Bounding $\gamma$ in \cref{cor il optimization incidence}, allows us to compare our result with the literature in integer linear optimization.
We obtain the following result.

\begin{corollary}[Integer linear optimization -- incidence treewidth v2]
\label{cor il optimization incidence v2}
Consider a linear inequality system $S_L = \pare{\vars,\domI,\consLle,\consLge}$, let $\consL := \consLle \cup \consLge$, and let $T$ be a tree decomposition of $G_{\inc}\pare{S_L}$ of treewidth $\itwidth$.
Let $\nu_\var \in \R$ for every $\var \in \vars$.
Let $A_{\max}$ be the maximum absolute value of $a^\con_\var$, for $\var \in \vars$ and $\con \in \consL$.
Then, the integer linear optimization problem can be solved in time
\eqref{eq il opt runtime}, where $\pwidth$ is replaced by $\pare{2A_{\max} \dom_{\max} \card{\vars}+1}^{\itwidth+1}.$
\end{corollary}

\begin{prf}
We apply \cref{cor il optimization incidence}.
To bound $\gamma := \max\bra{\gamma^\con \mid \con \in \consL}$, recall (see \cref{sec def projection-width}) that 
$\gamma^\con$, for $\con \in \consL$, is defined by:
\begin{align*}
& \gamma^\con := 
\delta^\con - \sum_{\var \in \vars} \min\bra{a^\con_\var d' \mid d' \in \domI}.
\end{align*}
As discussed in \cref{sec def projection-width}, we can assume $\gamma^\con \ge 0$ for every $\con \in \consL$.
We can also assume $\delta^\con \le A_{\max} \dom_{\max} \card{\vars}$ for every $\con \in \consL$.
For $\con \in \consLle$, this is because otherwise all assignments from $\vars$ to $\domI$ satisfy the constraint $\con$.
For $\con \in \consLge$, this is because otherwise no assignment from $\vars$ to $\domI$ satisfies the constraint $\con$.
We then obtain $\gamma \le 2 A_{\max} \dom_{\max} \card{\vars}$.
\end{prf}

We observe that \cref{cor il optimization incidence v2} recovers the tractability result for integer linear optimization established in \cite{GanOrdRam17}.
In fact, our running time is comparable to that of theorem~11 in~\cite{GanOrdRam17},
which is
\begin{align*}
O(\pare{A_{\max} \dom_{\max} \card{\vars}}^{2\itwidth+2} \itwidth (\card{\vars}+\card{\consL})).
\end{align*}

\subsubsection{Binary polynomial optimization}
\label{sec BPO itw}


For a hypergraph $H$,
it is shown in \cite{CapdPDiG24xx} that a tree decomposition of $G_{\inc}\pare{H}$ of treewidth $\itwidth$ yields a tree decomposition of $G_{\inc}\pare{S_{\BPO}}$ of treewidth at most $2(\itwidth+1)$.
Combining this result with \cref{cor BPO,lem itw PSw} gives the following.

\begin{corollary}[Binary polynomial optimization -- incidence treewidth]
\label{cor BPO itw}
Consider a hypergraph $H=(V,E)$ and let $T$ be a tree decomposition of $G_{\inc}\pare{H}$ of treewidth $\itwidth$.
Let $\val_v \in \R$ for every $v \in V$, and $\val_e \in \R$ for every $e \in E$.
Then, the binary polynomial optimization problem can be solved in time 
\begin{align*}
O\pare{2^{6\itwidth+9} \pare{\card{V} + \size(H)} \size(H)},
\end{align*}
\end{corollary}

\cref{cor BPO itw} recovers the tractability result in \cite{CapdPDiG24xx} for binary polynomial optimization over hypergraphs with bounded incidence treewidth.
Unlike the proof in \cite{CapdPDiG24xx}, ours does not rely on knowledge compilation.

For constrained binary polynomial optimization, we can employ \cref{cor optimization itw,lem itw PSw} to obtain the following tractability result.

\begin{corollary}[Constrained binary polynomial optimization -- incidence treewidth]
\label{cor constrained BPO itw}
Consider a hypergraph $H=(V,E)$ and let $T$ be a tree decomposition of $G_{\inc}\pare{H}$ of treewidth $\itwidth$.
Consider a separable system $S_c = \pare{\vars_V \cup \yvars_E,\bra{0,1},\consinc,\consgec}$ and let $W$ be a vertex cover of $G_{\inc}\pare{S_c}$ of cardinality $\kappa$.
Let $\val_v \in \R$ for every $v \in V$, and $\val_e \in \R$ for every $e \in E$.
Then, the constrained binary polynomial optimization problem can be solved in time
\begin{align*}
O\pare{2^{6\itwidth+9+3\kappa} \pare{\card{V} + \size(H)} \size(H)}.
\end{align*}
\end{corollary}

\begin{prf}
From the proof of lemma~3 in \cite{CapdPDiG24xx}, a tree decomposition of $G_{\inc}\pare{H}$ of treewidth $\itwidth$ yields a tree decomposition of $G_{\inc}\pare{S_{\BPO}}$ of treewidth at most $2(\itwidth+1)$.
Now denote by $S'$ the separable system obtained by putting together the separable systems $S_{\BPO}$ and $S_c$, that is, $S' = \pare{\vars_V \cup \yvars_E,\bra{0,1},\consinc,\consBPO \cup \consgec}$.
Note that $G_{\inc}\pare{S'}$ is obtained by adding isolated vertices to $G_{\inc}\pare{S_{\BPO}}$ and $G_{\inc}\pare{S_c}$, and then taking their union.
Clearly, adding isolated vertices does not increase the treewidth or the cardinality of a vertex cover.
It then follows from lemma~3 in~\cite{AleLozQuiRabRazZam24}, that $G_{\inc}\pare{S'}$ has treewidth at most $2(\itwidth+1)+\kappa$.
The result then follows from \cref{cor optimization itw}.
\end{prf}

\cref{cor constrained BPO itw} is just one illustration of how \cref{cor optimization itw} can be applied to constrained binary polynomial optimization.
Other results of the same type can be obtained by replacing lemma~3 in \cite{AleLozQuiRabRazZam24} with analogous bounds on the incidence treewidth of the combined system.
In particular, \cref{cor constrained BPO itw} is new and substantially extends the recent result of \cite{CapdPDiG24xx}, which establishes tractability of binary polynomial optimization with a single ``extended cardinality constraint.''
More general tractability statements can be obtained by applying directly \cref{th optimization} instead of \cref{cor optimization itw}.

\bigskip

\begin{small}
%

\noindent
\textbf{Funding.} Alberto~Del~Pia is partially funded by AFOSR grant FA9550-23-1-0433 and ONR grant N00014-25-1-2490. Any opinions, findings, and conclusions or recommendations expressed in this material are those of the author and do not necessarily reflect the views of the Air Force Office of Scientific Research or of the Office of Naval Research.
\end{small}

\ifthenelse {\boolean{MPA}}
{
\bibliographystyle{spmpsci}
}
{
\bibliographystyle{alpha}
}


\newcommand{\etalchar}[1]{$^{#1}$}


\end{document}